\documentclass[11pt]{amsart}

\usepackage{amsmath}

\usepackage{amssymb}

\usepackage{amsthm}

\usepackage{color}

\newtheorem{thm}{Theorem}[section]

\newtheorem{lem}{Lemma}[section]

\newtheorem{rem}{Remark}[section]

\newcommand{\goto}{\rightarrow}

\def\<{\left<}\def\>{\right>}
\def\({\left(}\def\){\right)}

\allowdisplaybreaks

\begin{document}

\title{A draw-down  reflected spectrally negative L\'{e}vy   process}
\author{Wenyuan Wang and  Xiaowen Zhou}

\address{Wenyuan Wang: School of Mathematical Sciences, Xiamen University, Fujian 361005, People's Republic of China;
Email address: wwywang@xmu.edu.cn.
Xiaowen Zhou: Department of Mathematics and Statistics, Concordia University, Montreal, Quebec, Canada H3G 1M8; Email address:  xiaowen.zhou@concordia.ca.}


\begin{abstract}
In this paper we study  a spectrally negative L\'{e}vy process that is reflected at its draw-down level whenever  a draw-down time  from the running supremum arrives. Using an excursion-theoretical approach, for such a reflected process we find the Laplace transform of the upper exiting time  and an expression of the associated potential measure.   When the reflected process is identified as a risk process with capital injections,
the expected total amount of discounted capital injections prior to the  exiting time and the Laplace transform of the accumulated capital injections until the  exiting time  are also obtained. The results are expressed in terms of scale functions for spectrally negative L\'evy processes.
\end{abstract}

\maketitle

\section{Introduction }
\setcounter{section}{1}

The first passage problems have been studied extensively for spectrally negative L\'evy processes in recent years. Such problems often concern the Laplace transforms for quantities associated to the exit times, the potential measures and the weighted occupation times. Using Wiener-Hopf factorization and excursion theory,  those Laplace transforms can often be expressed semi-explicitly  using scale functions for the spectrally negative L\'evy processes. We refer to Bertoin (1996) and Kyprianou (2006) and references therein for results along this line.

 A draw-down time is the first downward passage time from a level that depends on the previous supremum of the process via the so called draw-down function. It was first introduced and studied for diffusion processes in Lehoczky (1977).  More recently, progress has been made in investigating the general draw-down times for spectrally negative L\'evy processes. In Avram,  Vu and Zhou (2019)  the linear draw-down time related two-sided exit problem  is solved for a spectrally negative L\'evy type risk process. More fluctuation results for general draw-down times are obtained via excursion theory arguments in Li, Vu and Zhou (2019). A draw-down time based dividend optimization is further considered in Wang and Zhou (2018).

The  spectrally negative L\'evy process reflected either from its running supremum or from its running infimum frequently appears in a wide
variety of applications, such as the study of the water level for a dam, the queueing theory (cf., Asmussen (1989), Borovkov (1976) and Prabhu (1997)), the optimal stopping problems (cf., Baurdoux and Kyprianou (2008) and Shepp and Shiryaev (1994)) and the optimal control problems (cf., Avram et al. (2007), De Finetti (1957) and Gerber (1990)) for L\'evy risk processes. We refer to Pistorius (2004, 2007), Zhou (2007) and Kyprianou (2006) and references therein for a collection of results on reflected spectrally negative L\'evy processes.

Given the previous results on the draw-down times and on the reflected processes for spectrally negative L\'evy processes, it is natural to introduce draw-down related reflected processes. The main purpose of this paper is to propose a new process that is obtained by reflecting the spectrally negative L\'evy process from consecutive draw-down levels. Intuitively, given a draw-down function this process first evolves like a spectrally negative L\'evy process until immediately before the draw-down time when it is to down-cross the associated draw-down level.   After this moment, it starts to evolve according to a spectrally negative L\'evy process reflected at the draw-down level until it comes back to its  historical high, after which time the process repeats the previous behavior   all over again for the updated draw-down times and draw-down levels.     It naturally generalizes the classical reflected process from the running infimum process  to the reflected process whose reflecting levels depend on the previous supremums of the process.
In this paper we first develop some new fluctuation identities for the draw-down reflected process, which generalize those for the spectrally negative L\'evy process reflected at its infimum.

Spectrally negative L\'evy processes are often used in risk theory to model the surplus processes. The reflected process at the infimum has the interpretation of a surplus process with capital injections;  see  Dickson and  Waters (2004) and Avram et al. (2007) for some earlier work on risk models with capital injections.
The draw-down time can be treated as a generalized ruin time that depends on the historical high of the surplus.
The draw-down reflected process can thus be identified as a L\'evy risk model with capital injections  to keep the surplus above the respective draw-down levels so that the net drops of the surplus from its historical highs are kept within certain ranges that can also depend on the historical highs. In this paper we also carry out capital injections related computations that are interesting in risk theory.

Since the spectrally negative L\'evy process reflected at its running supremum  is a Markov process and the running supremum process is a version of the local time at $0$ for the reflected process, the fluctuation behaviors of the underlying L\'evy process can often be described by   the Poisson point  process of  excursions away from the running supremum. The desired results then follow from  excursion theory techniques such as compensation formulas.
Using the excursion-theoretical approach, Kyprianou and Pistorius (2003)
derived the Laplace transform of a first passage time which is the key to the evaluation of the
Russian option; Avram et al. (2004) determined the joint Laplace transform of the exit time and exit position from an interval containing the origin of the process reflected at its supremum, which is then applied to solve the optimal stopping problems associated with the pricing of Russian options and their Canadized versions; Pistorius (2004) derived the $q$-resolvent kernels for the L\'{e}vy process reflected at its supremum  killed upon leaving $[0, a]$; Pistorius (2007) solved the problem of the Lehoczky and Skorokhod embedding problem for the the spectrally negative L\'{e}vy process reflected at its supremum;
Baurdoux (2007) investigated the density of the resolvent measure of the killed L\'evy process reflected at its infimum;
Kyprianou and Zhou (2009) obtained  the Gerber-Shiu function for a generalized L\'evy risk process.

The excursion theory also plays a key role in obtaining the results of  this paper. Using the excursion approach,  some classical results on the spectrally negative L\'evy process reflected from the infimum are generalized to the process with draw-down reflection. In particular,  we obtain the Laplace transform for the upward exit time and the potential measure for such a draw-down reflected process.
 We also find expressions on the expected present value of cumulated amount of capital injections up to an upward exit time and the Laplace transform for the total amount of capital injections until the exit time for the associated L\'evy risk process. These results are expressed in terms of scale functions for the spectrally negative L\'evy process. When the  general supremum dependent draw-down time  is reduced to
the downward first passage time of a constant boundary, we are able to recover the corresponding classical results in the existing literature.

The rest of the paper is arranged as follows.
In Section 2 we first present some preliminary results concerning the spectrally negative L\'{e}vy process and  its reflection from below at a fixed  level, and  then define the general draw-down reflected spectrally negative L\'{e}vy process. The associated excursion process of excursions from the supremum is also introduced in this section.
The main results and their proofs are provided in Section 3. Some technical lemmas and discussions are also included in this section.

\section{Spectrally negative L\'{e}vy process and its reflected processes}
\setcounter{section}{2}

Write $X\equiv\{X(t);t\geq0\}$, defined on a probability space with probability  laws $\{\mathbb{P}_{x};x\in(-\infty,\infty)\}$ and natural filtration $\{\mathcal{F}_{t};t\geq0\}$, for a spectrally negative L\'{e}vy process that is not a purely increasing linear drift or the negative of a subordinator. Denote its running supremum process as $\overline{X}\equiv\{\sup\limits_{0\leq s\leq t}X(s),\,t\geq0\}$ with $\overline{X}(
0)=x$ under $\mathbb{P}_{x}$.
Given a value $a$, the process $X$ reflected from below at the level $a$ is defined as
\[X(t)-(\underline{X}(t)-a)\wedge 0, \,\, t\geq 0 \]
where $\underline{X}(t):=\inf_{0\leq s\leq t}X(s)$ with $\underline{X}(
0)=x$ under $\mathbb{P}_{x}$, denotes the running infimum process.
Let $\{Y(t), t\geq0\}$ be the process $X$ reflected from below at the level $0$ (cf., Pistorius (2004)).

The  draw-down time associated to a draw-down function $\xi$ on $(-\infty,\infty)$ satisfying $\xi(x)<x, \,\, x\in(-\infty,\infty)$,  the $\xi$-draw-down time in short, is defined as
$$\tau_{\xi}\equiv \tau_{\xi}(X):=\inf\{t\geq0: X(t)<\xi(\overline{X}(t))\}$$
with the convention $\inf\emptyset:=\infty$.
 We define the process $X$ reflected at the $\xi$-draw-down time $\tau_{\xi}$ as
  \[X(t)-\mathbf{1}_{[\tau_\xi,\infty)}(t)\left(\inf_{\tau_\xi\leq s\leq t}{X}(s)- \xi(\overline{X}(\tau_\xi))\right)\wedge 0, \,\, t\geq 0, \]
where we call $\xi(\overline{X}(\tau_\xi))$ the draw-down level at the draw-down time $\tau_\xi$.


We now define the draw-down reflected process $U$ for $X$.
Intuitively, the process $U$ initially agrees with $X$ until the first draw-down time of $U$. Then it starts to evolve according to $X$ reflected at the draw-down level until the next draw-down time of $U$ when it is reflected at the draw-down level again, and so on. Then given that $U(s)=\overline{U}(s):=\sup_{0\leq t\leq s}U(t)$, the process $\{U(t);t\geq s\}$ evolves without reflection until the next draw-down time $\tau_\xi$; and given that $\overline{U}(s)> U(s)$,
the process $\{U(t); t\geq s\}$ is reflected from below at the current draw-down level $\xi(\overline{U}(s))$ until it comes back to the level  $\overline{U}(s)$. Note that the process $U$ is not a Markov process in general, but the process $(U,\overline{U})$ is Markovian.
Write $\mathbb{P}_{x,y}$ and $\mathbb{E}_{x,y}$ for the law of $(U,\overline{U})$ such that $U(0)=x$ and $\overline{U}(0)=y$. For simplicity, denote $\mathbb{P}_{x}=\mathbb{P}_{x,x}$ and $\mathbb{E}_{x}=\mathbb{E}_{x,x}$.

To be more precise, define $T_0:=0$ and $U(T_0):=X(0) $.
 Suppose first that  for $n\geq 1$, $U(t)$ has been defined on $[0, T_n]$ for $T_n<\infty$, $n\geq 1$.
 Let $X_{n+1}$ be an independent copy of $X$ starting at $U(T_n)$ and $U_{n+1}$ be the process $X_{n+1}$ reflected at its $\xi$-draw-down time $\tau_{\xi}({X}_{n+1})$. If $\tau_{\xi}(X_{n+1})=\infty$, let $T_{n+1}:=\infty$, and if $\tau_{\xi}(X_{n+1})<\infty$, let
$$T_{n+1}:=T_n+\inf\{t\geq 0: U_{n+1}(t)>\overline{X}_{n+1}(\tau_{\xi}(X_{n+1})) \}, $$
 where $\overline{X}_{n+1}(t):=\sup\limits_{0\leq s\leq t}X_{n+1}(s)$. Observe that $T_{n+1}<\infty$ if $\tau_{\xi}(X_{n+1})<\infty $.
Then define
\[U(T_n+t):=U_{n+1}(t) \quad\text{for}\quad t\in [0, T_{n+1}-T_n) \quad\text{and}\quad U(T_{n+1}):=U_{n+1}(T_{n+1}-T_n)\quad\text{if}\quad T_{n+1}<\infty.\]
Suppose now that $U(t)$ has been defined on $[0, T_n=\infty)$ for $ n\geq 0$.
For convenience, let $T_{n+1}:=\infty$.
One can show that $T_n\uparrow\infty$ as $n\goto\infty$ under mild conditions on $\xi$; see Lemma \ref{lemma0}.

For the process $X$, define its first up-crossing time of level $b\in(-\infty,\infty)$ and first down-crossing time of level $c\in(-\infty,\infty)$, respectively, by
\begin{eqnarray}
\tau^{+}_{b}:=\inf\{t\geq0: X(t)>b\}\,\,\, \text{and}\,\,\, \tau_{c}^{-}:=\inf\{t\geq0: X(t)<c\}.\nonumber
\end{eqnarray}

For the processes $Y$ and $U$, their first up-crossing times of $b\in(-\infty,\infty)$ are defined respectively by
\begin{eqnarray}
\sigma^{+}_{b}:=\inf\{t\geq0: Y(t)>b\} \,\,\, \text{and} \,\,\,\kappa_{b}^{+}:=\inf\{t\geq0: U(t)>b\}.\nonumber
\end{eqnarray}

Let the Laplace exponent of $X$ be given by
\begin{eqnarray}
\psi(\theta):=\ln \mathbb{E}_{x}\left(\mathrm{e}^{\theta (X_{1}-x)}\right)=\gamma\theta+\frac{1}{2}\sigma^{2}\theta^{2}-\int_{(0,\infty)}\left(1-\mathrm{e}^{-\theta x}-\theta x\mathbf{1}_{(0,1)}(x)\right)\nu(\mathrm{d}x),\nonumber
\end{eqnarray}
where $\nu$ is  the L\'{e}vy measure satisfying $\int_{(0,\infty)}\left(1\wedge x^{2}\right)\nu(\mathrm{d}x)<\infty$.
It is known that $\psi(\theta)$ is finite for  $\theta\in[0,\infty)$ in which case it is strictly convex and infinitely differentiable.
As in Bertoin (1996), the $q$-scale functions $\{W^{(q)};q\geq0\}$ of $X$ are defined as follows. For each $q\geq0$, $W^{(q)}:\,[0,\infty)\rightarrow[0,\infty)$ is the unique strictly increasing and continuous function with Laplace transform
\begin{eqnarray}
\int_{0}^{\infty}\mathrm{e}^{-\theta x}W^{(q)}(x)\mathrm{d}x=\frac{1}{\psi(\theta)-q},\quad \mbox{for }\theta>\Phi(q),\nonumber
\end{eqnarray}
where $\Phi(q)$ is the largest solution of the equation $\psi(\theta)=q$. Further define $W^{(q)}(x)=0 $ for $x<0$, and write $W$ for the $0$-scale function $W^{(0)}$.

It is known that $W^{(q)}(0)=0 $ if and only if process $X$ has sample paths of unbounded variation.
If $X$ has sample paths of unbounded variation, or if $X$ has sample paths of bounded variation and the L\'{e}vy measure has no atoms, then the scale function $W^{(q)}$ is continuously differentiable over $(0, \infty)$.
By Loeffen (2008), if $X$ has a L\'{e}vy measure which has a completely monotone density, then $W^{(q)}$ is twice continuously differentiable over $(0, \infty)$ when $X$ is of unbounded variation.
Moreover, if process $X$ has a nontrivial Gaussian component, then $W^{(q)}$ is twice continuously differentiable over $(0, \infty)$.
The interested readers are referred to Chan et al. (2011) and Kuznetsov et al. (2012) for more detailed discussions on the smoothness of scale functions.
For results on numerical computation of the scale function, the readers are referred to Hubalek and Kyprianou (2011) and the references therein.

 Further define
\begin{eqnarray}
Z^{(q)}(x):=1+q\int_{0}^{x}W^{(q)}(z)\mathrm{d}z,\quad x\geq0,\nonumber
\end{eqnarray}
and
\begin{eqnarray}
Z^{(q)}(x,\theta):=\mathrm{e}^{\theta x}\left(1-\left(\psi(\theta)-q\right)\int_{0}^{x}\mathrm{e}^{-\theta z}W^{(q)}(z)\mathrm{d}z\right),\quad \theta\geq0, \,x\geq0,\nonumber
\end{eqnarray}
with $Z(x,\theta):=Z^{(0)}(x,\theta)$,
and
$$\overline{W}^{(q)}(x):=\int_{0}^{x}W^{(q)}(z)\mathrm{d}z,\quad q\geq0, x\geq0,$$
and
$$\overline{Z}^{(q)}(x):=\int_{0}^{x}Z^{(q)}(z)\mathrm{d}z=x+q\int_{0}^{x}\int_{0}^{z}W^{(q)}(w)\mathrm{d}w\mathrm{d}z,\quad q\geq0, x\geq0.$$

In the sequel, without loss of generality we assume $X_{1}\equiv X$. By Li et al. (2017), we have
\begin{eqnarray}\label{part1}
\mathbb{E}_{x}(\mathrm{e}^{-q\kappa_{b}^{+}}\mathbf{1}_{\{\kappa_{b}^{+}<\tau_{\xi}\}})=
\mathbb{E}_{x}\left(\mathrm{e}^{-q\tau_{b}^{+}}\mathbf{1}_{\{\tau_{b}^{+}<\tau_{\xi}\}}\right)
=\exp\left(-\int_{x}^{b}\frac{W^{(q)\prime}(\overline{\xi}\left(z\right))}
{W^{(q)}(\overline{\xi}\left(z\right))}\mathrm{d}z\right),
\end{eqnarray}
where $\overline{\xi}(z)=z-\xi(z)$.
For $x\in[0,b]$ and $q\geq0$, from Proposition 2 in Pistorius (2004) we have
\begin{eqnarray}\label{two.sid.exit.Y}
\mathbb{E}_{x}(\mathrm{e}^{-q\sigma^{+}_{b}})=\frac{Z^{(q)}(x)}{Z^{(q)}(b)}.
\end{eqnarray}

By Kyprianou (2006), the resolvent measure corresponding to $X$ is absolutely continuous with respect to the Lebesgue measure with  density given by
\begin{eqnarray}\label{h2}
\hspace{-0.3cm}&&\hspace{-0.3cm}\int_{0}^{\infty}\mathrm{e}^{-qt}\mathbb{P}_{x}(X(t)\in \mathrm{d}y;t<\tau_{c}^{-}\wedge \tau_{b}^{+})\mathrm{d}t
\nonumber\\
\hspace{-0.3cm}&=&\hspace{-0.3cm}
\left(\frac{W^{(q)}(x-c)}{W^{(q)}(b-c)}W^{(q)}(b-y)-W^{(q)}(x-y)\right)\mathbf{1}_{(c,b)}(y)\mathrm{d}y,
\end{eqnarray}
for $x\in(c, b)$.
By Pistorius (2004), the resolvent measure corresponding to $Y$ is also absolutely continuous with respect to the Lebesgue measure and has a version of density given by
\begin{eqnarray}\label{reso.meas.Y}
\hspace{-0.3cm}&&\hspace{-0.3cm}\int_{0}^{\infty}\mathrm{e}^{-qt}\mathbb{P}_{x}(Y(t)\in \mathrm{d}y,t<\sigma^{+}_{b})\mathrm{d}t
\nonumber\\
\hspace{-0.3cm}&=&\hspace{-0.3cm}\left(\frac{Z^{(q)}(x)}{Z^{(q)}(b)}W^{(q)}(b-y)-W^{(q)}(x-y)\right)
\mathbf{1}_{[0,b)}(y)\mathrm{d}y,
\end{eqnarray}
where $x\in[0, b)$.

Define the total amount of capital injections made until time $t$ for the draw-down reflected process  as
\begin{eqnarray}
R(t)\hspace{-0.3cm}&:=&\hspace{-0.3cm}
-\sum_{k=1}^{N-1}\mathbf{1}_{[T_{k-1}+\tau_{\xi}(X_k),\infty)}(t)\left(   \inf_{\tau_\xi(X_k)\leq s\leq T_k\wedge t-T_{k-1}}{X}_{k}(s)-\xi(\overline{X}_{k}(\tau_{\xi}(X_k)))\right)\wedge 0.
\nonumber
\end{eqnarray}
where $N:=\inf\{n: T_n=\infty\}=\inf\{n: \tau_\xi(X_{n})=\infty\}$.
Then the expectation of the total discounted capital injections until $\kappa_{b}^{+}$ is defined by
\begin{eqnarray}
V_{\xi}(x;b)\hspace{-0.3cm}&:=&\hspace{-0.3cm}\mathbb{E}_{x}\left(\int_{0}^{\kappa_{b}^{+}}\mathrm{e}^{-q t}\mathrm{d}R(t)\right),\quad b\geq x,\nonumber
\end{eqnarray}
and the Laplace transform of the total non-discounted capital injection until $\kappa_{b}^{+}$ is defined by
\begin{eqnarray}
\overline{V}_\xi(x;b)
\hspace{-0.3cm}&
:=&\hspace{-0.3cm}
\mathbb{E}_{x}\left(\mathrm{e}^{-\theta R(\kappa_{b}^{+})}\right),\quad b\geq x
.\nonumber
\end{eqnarray}

We also briefly recall concepts in excursion theory for the reflected process $\{\overline{X}(t)-X(t);t\geq0\}$, and we refer to Bertoin (1996) for more details.
For $x\in(-\infty,\infty)$, the process $\{L(t):= \overline{X}(t)-x, t\geq0\}$ serves as a local time at $0$ for
the Markov process $\{\overline{X}(t)-X(t);t\geq0\}$ under $\mathbb{P}_{x}$.
Let the corresponding inverse local time be defined as
$$L^{-1}(t):=\inf\{s\geq0: L(s)>t\}=\sup\{s\geq0: L(s)\leq t\}.$$
Further let $L^{-1}(t-):=\lim\limits_{s\uparrow t}L^{-1}(s)$.
Define a Poisson point process $\{(t, e_{t}); t\geq0\}$ as
$$e_{t}(s):=X(L^{-1}(t))-X(L^{-1}(t-)+s), \,\,s\in(0,L^{-1}(t)-L^{-1}(t-)],$$
whenever the lifetime of $e_{t}$ is positive, i.e. $L^{-1}(t)-L^{-1}(t-)>0$.
Whenever $L^{-1}(t)-L^{-1}(t-)=0 $, define $e_{t}:=\Upsilon$ with $\Upsilon$ being an additional isolated point.
A result of It\^{o} states that $e$ is a Poisson point process with
characteristic measure $n$
if $\{\overline{X}(t)-X(t);t\geq0\}$ is recurrent; otherwise $\{e_{t}; t\leq L(\infty)\}$ is a Poisson point process stopped at the first excursion of infinite lifetime. Here, $n$ is a measure on the space  $\mathcal{E}$ of excursions,
i.e. the space $\mathcal{E}$  of c\`{a}dl\`{a}g functions $f$ satisfying
\begin{eqnarray}
&&f:\,(0,\zeta)\rightarrow (0,\infty)\,\quad \mbox{for some } \zeta=\zeta(f)\in(0,\infty],
\nonumber\\
&&f:\,\{\zeta\}\rightarrow (0,\infty)\,\,\,\,\,\,\quad \mbox{if } \zeta<\infty,
\nonumber
\end{eqnarray}
where $\zeta=\zeta(f)$ is the excursion length or lifetime; see Definition 6.13 of Kyprianou (2006) for the definition of $\mathcal{E}$.
Denote by $\varepsilon(\cdot)$, or $\varepsilon$ for short, a generic excursion
belonging to the space $\mathcal{E}$ of canonical excursions.
The excursion height of a canonical excursion $\varepsilon$ is
denoted by $\overline{\varepsilon}=\sup\limits_{t\in[0,\zeta]}\varepsilon(t)$. The first passage time of a canonical excursion $\varepsilon$ is defined
by
$$
\rho_{b}^{+}\equiv\rho_{b}^{+}(\varepsilon) :=\inf\{t\in[0,\zeta]: \varepsilon(t)>b\},
$$
with the convention $\inf\emptyset:=\zeta$.

Denote by $\varepsilon_{g}$
the excursion (away from $0$)  with left-end point $g$ for the reflected process $\{\overline{X}(t)-X(t);t\geq0\}$, and $\zeta_{g}$ and $\overline{\varepsilon}_{g}$ denote its lifetime and excursion height, respectively; see Section IV.4 of Bertoin (1996).

\section{Main results}
\setcounter{section}{3}

In this section we present several results concerning the general draw-down reflected process  $U$. Recall $\overline{\xi}(x)=x-\xi(x)$. We first give the following Lemma \ref{lemma0}  guaranteeing  the well-definedness  of the process $U$.

\vspace{0.2cm}
\begin{lem}\label{lemma0}
Given $x\in(-\infty,\infty)$, if $\overline{\xi}$ is bounded from below on $[x, \infty)$, i.e.
$$\alpha:=\inf\limits_{y\in[x,\infty)}\overline{\xi}(y)>0,$$
we have $\mathbb{P}_{x}\left(\lim\limits_{n\rightarrow\infty}T_{n}=\infty\right)=1$.
\end{lem}

\begin{proof}[Proof:]\,\,\,For $q\in(0,\infty)$, by the strong Markov property of $(U,\overline{U})$ and \eqref{two.sid.exit.Y}, one gets
\begin{eqnarray}
\label{T1}
\mathbb{E}_{z}\left(\mathrm{e}^{-q T_{1}}\right)
\hspace{-0.3cm}&=&\hspace{-0.3cm}
\mathbb{E}_{z}\left(\mathrm{e}^{-q \tau_{\xi}}\mathbf{1}_{\{\tau_{\xi}<\infty\}}\,\mathbb{E}_{z}\left(\left.\mathrm{e}^{-q \,(T_{1}-\tau_{\xi})}\right|\mathcal{F}_{\tau_{\xi}}\right)\right)
\nonumber\\
\hspace{-0.3cm}&=&\hspace{-0.3cm}
\mathbb{E}_{z}\left(\mathrm{e}^{-q \tau_{\xi}}\mathbf{1}_{\{\tau_{\xi}<\infty\}}\,
\left[\left.\mathbb{E}_{}\left(\mathrm{e}^{-q \sigma_{z}^{+}}\right)\right|_{z=\overline{\xi}\left(\overline{X}(\tau_{\xi})\right)}\right]\right)
\nonumber\\
\hspace{-0.3cm}&=&\hspace{-0.3cm}
\mathbb{E}_{z}\left(\frac{\mathrm{e}^{-q \tau_{\xi}}\mathbf{1}_{\{\tau_{\xi}<\infty\}}}{Z^{(q)}\left(\overline{\xi}\left(\overline{X}(\tau_{\xi})\right)\right)}
\right)
\nonumber\\
\hspace{-0.3cm}&\leq&\hspace{-0.3cm}
\frac{1}{Z^{(q)}\left(\alpha\right)},\quad z\in[x,\infty),
\end{eqnarray}
where, $\overline{X}(\tau_{\xi})\geq z\geq x$ and $\mathrm{e}^{-q \tau_{\xi}}\leq 1$\, $\mathbb{P}_{z}$-a.s. for $\tau_\xi<\infty$, the definition of $\alpha$ and the increasing property of $Z^{(q)}$ are used in the inequality. Hence, by \eqref{T1} and $U(T_{n-1})\geq x$\, $\mathbb{P}_{x}$-a.s. for $T_{n-1}<\infty$, one has
\begin{eqnarray}\label{Tn}
\mathbb{E}_{x}\left(\mathrm{e}^{-q T_{n}}\right)
\hspace{-0.3cm}&=&\hspace{-0.3cm}
\mathbb{E}_{x}\left(\mathrm{e}^{-q T_{n-1}}\mathbf{1}_{\{T_{n-1}<\infty\}}\,\mathbb{E}_{x}\left(\left.\mathrm{e}^{-q \,(T_{n}-T_{n-1})}\right|\mathcal{F}_{T_{n-1}}\right)\right)
\nonumber\\
\hspace{-0.3cm}&=&\hspace{-0.3cm}
\mathbb{E}_{x}\left(\mathrm{e}^{-q T_{n-1}}\mathbf{1}_{\{T_{n-1}<\infty\}}\,\mathbb{E}_{U(T_{n-1})}\left(\mathrm{e}^{-q T_{1}}\right)
\right)
\nonumber\\
\hspace{-0.3cm}&\leq&\hspace{-0.3cm}
\frac{\mathbb{E}_{x}\left(\mathrm{e}^{-q T_{n-1}}\right)}{Z^{(q)}\left(\alpha\right)},\quad n\geq 2.
\end{eqnarray}
By recursively using \eqref{T1} and \eqref{Tn}, one can derive
\begin{eqnarray}\label{}
\mathbb{E}_{x}\left(\mathrm{e}^{-q T_{n}}\right)
\hspace{-0.3cm}&\leq&\hspace{-0.3cm}
\left(\frac{1}{Z^{(q)}\left(\alpha\right)}\right)^{n},\quad n\geq 1.\nonumber
\end{eqnarray}
Thus, we have
\begin{eqnarray}\label{}
\lim_{n\rightarrow\infty}\mathbb{E}_{x}\left(\mathrm{e}^{-q T_{n}}\right)
\hspace{-0.3cm}&=&\hspace{-0.3cm}
0.\nonumber
\end{eqnarray}
Then $\lim_{n\rightarrow\infty}T_{n}=\infty$\, $\mathbb{P}_{x}$-a.s. since $T_n$ is increasing in $n$.
\end{proof}

\vspace{0.2cm}
In preparation for the proofs of Theorems \ref{3.1}-\ref{3.4} in the sequel, we need the following lemma
whose proof is similar to that of Proposition 3.1 in Li et al. (2019) and is  omitted.
\begin{lem}\label{lemma2}
For $\theta,\,q>0$, $x\leq b$ and measurable function $\phi:\,(-\infty,\infty)\rightarrow(-\infty,\infty)$, we have
\begin{eqnarray}\label{10}
\hspace{-0.3cm}&&\hspace{-0.3cm}
\mathbb{E}_{x}\left(\mathrm{e}^{-q \tau_{\xi}}\,\mathrm{e}^{\theta X(\tau_{\xi})}\,\phi\left(\overline{X}(\tau_{\xi})\right); \tau_{\xi}<\tau_{b}^{+}\right)
\nonumber\\
\hspace{-0.3cm}&=&\hspace{-0.3cm}
\int_{x}^{b}\phi\left(s\right)\mathrm{e}^{\theta \xi(s)}
\exp\left(-\int_{x}^{s}\frac{W^{(q)\prime}(\overline{\xi}\left(z\right))}
{W^{(q)}(\overline{\xi}\left(z\right))}\mathrm{d}z\right)
\nonumber\\
\hspace{-0.3cm}&&\hspace{-0.3cm}
\times\left(\frac{W^{(q)\prime}(\overline{\xi}(s))}{W^{(q)}(\overline{\xi}(s))}Z^{(q)}(\overline{\xi}(s),\theta )-\theta  Z^{(q)}(\overline{\xi}(s),\theta )-(q-\psi(\theta ))W^{(q)}(\overline{\xi}(s))\right)
\mathrm{d}s.
\end{eqnarray}
In particular, we have
\begin{eqnarray}\label{12}
\hspace{-0.3cm}&&\hspace{-0.3cm}
\mathbb{E}_{x}\left(\mathrm{e}^{-q \tau_{\xi}}\,\phi\left(\overline{X}(\tau_{\xi})\right); \tau_{\xi}<\tau_{b}^{+}\right)
=
\int_{x}^{b}\phi\left(s\right)
\exp\left(-\int_{x}^{s}\frac{W^{(q)\prime}(\overline{\xi}\left(z\right))}
{W^{(q)}(\overline{\xi}\left(z\right))}\mathrm{d}z\right)
\nonumber\\
\hspace{-0.3cm}&&\hspace{3.8cm}
\times\left(\frac{W^{(q)\prime}(\overline{\xi}(s))}{W^{(q)}(\overline{\xi}(s))}Z^{(q)}(\overline{\xi}(s))-qW^{(q)}(\overline{\xi}(s))\right)
\mathrm{d}s
,
\end{eqnarray}
and
\begin{eqnarray}\label{13}
\hspace{-0.3cm}&&\hspace{-0.3cm}
\mathbb{E}_{x}\left(\mathrm{e}^{\theta X(\tau_{\xi})}\,\phi\left(\overline{X}(\tau_{\xi})\right); \tau_{\xi}<\tau_{b}^{+}\right)
=
\int_{x}^{b}\phi\left(s\right)\mathrm{e}^{\theta \xi(s)}
\exp\left(-\int_{x}^{s}\frac{W^{\prime}(\overline{\xi}\left(z\right))}
{W^{}(\overline{\xi}\left(z\right))}\mathrm{d}z\right)
\nonumber\\
\hspace{-0.3cm}&&\hspace{3cm}
\times\left(\frac{W^{\prime}(\overline{\xi}(s))}{W^{}(\overline{\xi}(s))}Z^{}(\overline{\xi}(s),\theta )-\theta  Z^{}(\overline{\xi}(s),\theta )+\psi(\theta )W^{}(\overline{\xi}(s))\right)
\mathrm{d}s
,
\end{eqnarray}
and
\begin{eqnarray}\label{14}
\hspace{-0.3cm}&&\hspace{-0.3cm}
\mathbb{E}_{x}\left(\mathrm{e}^{-q \tau_{\xi}}\left(\xi\left(\overline{X}(\tau_{\xi})\right)-X(\tau_{\xi})\right); \tau_{\xi}<\tau_{b}^{+}\right)
=\int_{x}^{b}
\exp\left(-\int_{x}^{s}\frac{W^{(q)\prime}(\overline{\xi}\left(z\right))}
{W^{(q)}(\overline{\xi}\left(z\right))}\mathrm{d}z\right)
\nonumber\\
\hspace{-0.3cm}&&\hspace{0.3cm}
\times\left(Z^{(q)}(\overline{\xi}(s))-\psi^{\prime}(0+)W^{(q)}(\overline{\xi}(s))-\frac{\overline{Z}^{(q)}(\overline{\xi}(s))-\psi^{\prime}(0+)
\overline{W}^{(q)}(\overline{\xi}(s))}{W^{(q)}(\overline{\xi}(s))}W^{(q)\prime}(\overline{\xi}(s))\right)\mathrm{d}s
.
\end{eqnarray}
\end{lem}

\vspace{0.2cm}
We start with the Laplace transform of the upper exiting time for the process $U$.

\vspace{0.2cm}
\begin{thm}\label{3.1}
\label{Laplace.tra.ucro.}
For $q>0$ and $x\leq b$, we have
\begin{eqnarray}
\label{upper.lower.boun.resu.}
\mathbb{E}_{x}(\mathrm{e}^{-q\kappa_{b}^{+}})=
\exp\left(-\int_{x}^{b}\frac{qW^{(q)}(\overline{\xi}(z))}{ Z^{(q)}(\overline{\xi}(z))}\mathrm{d}z\right).
\end{eqnarray}
\end{thm}

\begin{proof}[Proof:]\,\,\,
Denote by $f(x)$ the left hand side of (\ref{upper.lower.boun.resu.}). We have
\begin{eqnarray}\label{part.}
f(x)=\mathbb{E}_{x}\left(\mathrm{e}^{-q\kappa_{b}^{+}}\mathbf{1}_{\{\kappa_{b}^{+}<\tau_{\xi}\}}\right)
+\mathbb{E}_{x}\left(\mathrm{e}^{-q\kappa_{b}^{+}}\mathbf{1}_{\{\tau_{\xi}<\kappa_{b}^{+}\}}\right)
\end{eqnarray}
Note that by definition,  $\tau_{\xi}<\kappa_{b}^{+}$ implies $\overline{X}(\tau_{\xi})<b$ which further implies $T_{1}<\kappa_{b}^{+}$. Hence, taking use of \eqref{two.sid.exit.Y} and \eqref{12} we get
\begin{eqnarray}\label{part2}
\hspace{-0.3cm}&&\hspace{-0.3cm}\mathbb{E}_{x}\left(\mathrm{e}^{-q\kappa_{b}^{+}}\mathbf{1}_{\{\tau_{\xi}<\kappa_{b}^{+}\}}\right)
=\mathbb{E}_{x}\left(\mathrm{e}^{-q\kappa_{b}^{+}}\mathbf{1}_{\{\tau_{\xi}<T_{1}<\kappa_{b}^{+}\}}
\right)
\nonumber\\
\hspace{-0.3cm}&=&\hspace{-0.3cm}
\mathbb{E}_{x}\left(\mathrm{e}^{-q\tau_{\xi}}\mathbf{1}_{\{\tau_{\xi}<\tau_{b}^{+}\}} \left[\left.\mathbb{E}_{}\left(\mathrm{e}^{-q \sigma^{+}_{z}}\right)\right|_{z=\overline{\xi}(\overline{X}(\tau_{\xi}))}\right]f(\overline{X}(\tau_{\xi}))\right)
\nonumber\\
\hspace{-0.3cm}&=&\hspace{-0.3cm}
\mathbb{E}_{x}\left(\mathrm{e}^{-q\tau_{\xi}}\mathbf{1}_{\{\tau_{\xi}<\tau_{b}^{+}\}} \frac{f(\overline{X}(\tau_{\xi}))}{Z^{(q)}(\overline{\xi}(\overline{X}(\tau_{\xi})))}\right)
\nonumber\\
\hspace{-0.3cm}&=&\hspace{-0.3cm}
\int_{x}^{b}f(s)
\exp\left(-\int_{x}^{s}\frac{W^{(q)\prime}(\overline{\xi}(z))}{ W^{(q)}(\overline{\xi}(z))}\mathrm{d}z\right)
\left(\frac{W^{(q)\prime}(\overline{\xi}(s))}{W^{(q)}(\overline{\xi}(s))}-\frac{qW^{(q)}(\overline{\xi}(s))}{Z^{(q)}(\overline{\xi}(s))}\right)\mathrm{d}s.
\end{eqnarray}
Combining (\ref{part1}), (\ref{part.}) and (\ref{part2}), we obtain
\begin{eqnarray}\label{f(x)}
f(x)\hspace{-0.3cm}&=&\hspace{-0.3cm}\exp\left(-\int_{x}^{b}\frac{W^{(q)\prime}(\overline{\xi}\left(z\right))}
{W^{(q)}(\overline{\xi}\left(z\right))}\mathrm{d}z\right)
\nonumber\\
\hspace{-0.3cm}&&\hspace{-0.3cm}
+\int_{x}^{b}f(s)
\exp\left(-\int_{x}^{s}\frac{W^{(q)\prime}(\overline{\xi}(z))}{ W^{(q)}(\overline{\xi}(z))}\mathrm{d}z\right)
\left(\frac{W^{(q)\prime}(\overline{\xi}(s))}{W^{(q)}(\overline{\xi}(s))}-\frac{qW^{(q)}(\overline{\xi}(s))}{Z^{(q)}(\overline{\xi}(s))}\right)
\mathrm{d}s,\quad x\leq b.
\end{eqnarray}
Taking derivative on both sides of (\ref{f(x)}) with respect to $x$, we have
\begin{eqnarray}\label{f'(x)}
f^{\prime}(x)\hspace{-0.3cm}&=&\hspace{-0.3cm}
\frac{W^{(q)\prime}(\overline{\xi}\left(x\right))}
{W^{(q)}(\overline{\xi}\left(x\right))}
f(x)
-f(x)\left(\frac{W^{(q)\prime}(\overline{\xi}(x))}{W^{(q)}(\overline{\xi}(x))}-\frac{qW^{(q)}(\overline{\xi}(x))}{Z^{(q)}(\overline{\xi}(x))}\right)
\nonumber\\
\hspace{-0.3cm}&=&\hspace{-0.3cm}
f(x)\frac{qW^{(q)}(\overline{\xi}(x))}{Z^{(q)}(\overline{\xi}(x))},
\quad x\leq b.
\end{eqnarray}
Solving (\ref{f'(x)}) we obtain
\begin{eqnarray}\label{fwithC}
f(x)
\hspace{-0.3cm}&=&\hspace{-0.3cm}
C+\exp\left(-\int_{x}^{b}\frac{qW^{(q)}(\overline{\xi}(z))}{Z^{(q)}(\overline{\xi}(z))}\mathrm{d}z\right),\quad x\leq b,
\end{eqnarray}
for some constant $C$.
The boundary condition $f(b)=1$ together with (\ref{fwithC}) yields (\ref{upper.lower.boun.resu.}).
\end{proof}

\begin{rem} If $\xi(x)=kx-d$ for some $k\in(-\infty,1)$ and $d\in (0,\infty)$, we have
\begin{eqnarray}
\label{}
\mathbb{E}_{x}(\mathrm{e}^{-q\kappa_{b}^{+}})
\hspace{-0.3cm}&=&\hspace{-0.3cm}
\exp\left(-\int_{x}^{b}\frac{qW^{(q)}((1-k)z+d)}{ Z^{(q)}((1-k)z+d)}\mathrm{d}z\right)
\nonumber\\
\hspace{-0.3cm}&=&\hspace{-0.3cm}
\exp\left(-\frac{1}{1-k}\int_{(1-k)x+d}^{(1-k)b+d}\frac{qW^{(q)}(z)}{ Z^{(q)}(z)}\mathrm{d}z\right)
\nonumber\\
\hspace{-0.3cm}&=&\hspace{-0.3cm}
\left(\frac{Z^{(q)}((1-k)x+d)}{ Z^{(q)}((1-k)b+d)}\right)^{\frac{1}{1-k}}.\nonumber
\end{eqnarray}
\end{rem}

\vspace{0.2cm}
We then obtain an expression of the resolvent density for the process $U$.

\vspace{0.2cm}
\begin{thm}\label{3.2}
\label{reso.meas.U.}
For $q>0$, $x\leq b$ and $u\leq b$, the resolvent measure of $U$ is absolutely continuous with respect to the Lebesgue measure with  density given by
\begin{eqnarray}
\label{resovent.meas.}
\hspace{-0.3cm}&&\hspace{-0.3cm}
\int_{0}^{\infty}\mathrm{e}^{-qt}\mathbb{P}_{x}(U(t)\in \mathrm{d}u,t<\kappa_{b}^{+})\mathrm{d}t
\nonumber\\
\hspace{-0.3cm}&=&\hspace{-0.3cm}
W^{(q)}(0)\exp\left(-\int_{x}^{u}\frac{qW^{(q)}(\overline{\xi}(z))}{ Z^{(q)}(\overline{\xi}(z))}\mathrm{d}z\right)\mathbf{1}_{(x,b)}(u)\mathrm{d}u
+\int_{x}^{b}\exp\left(-\int_{x}^{y}\frac{qW^{(q)}(\overline{\xi}(z))}{ Z^{(q)}(\overline{\xi}(z))}\mathrm{d}z\right)
\nonumber\\
\hspace{-0.3cm}&&\hspace{-0.3cm}
\times
\left(W^{(q)\prime}(y-u)-\frac{qW^{(q)}(\overline{\xi}(y))}{Z^{(q)}(\overline{\xi}(y))}W^{(q)}(y-u)\right)
\mathbf{1}_{(\xi(y),y)}(u)\mathrm{d}y  \mathrm{d}u.
\end{eqnarray}
\end{thm}

\begin{proof}[Proof:]\,\,\,
Recall $\overline{U}(t)=\sup_{s\in[0,t]}U(s)$ and let $e_{q}$ be an exponential  random variable independent of $X$.
For $q>0$, $x\leq b$ and any continuous, non-negative and bounded function $h$, let
\begin{eqnarray}
\label{gene.reso.meas.}
qg(x)\hspace{-0.3cm}&:=&\hspace{-0.3cm}\int_{0}^{\infty}q \mathrm{e}^{-qt}\mathbb{E}_{x}(h(U(t));t<\kappa_{b}^{+})\mathrm{d}t
\nonumber\\
\hspace{-0.3cm}&=&\hspace{-0.3cm}
\mathbb{E}_{x}\left(h(X(e_{q}))\mathbf{1}_{\{X(e_{q})<\overline{X}(e_{q}),\,e_{q}<\tau_{b}^{+}\wedge \tau_{\xi}\}}\right)
+\mathbb{E}_{x}\left(h(U(e_{q}))\mathbf{1}_{\{U(e_{q})<\overline{U}(e_{q}),\tau_{\xi}<e_{q}<\kappa_{b}^{+}\}}\right)
\nonumber\\
\hspace{-0.3cm}&&\hspace{-0.3cm}
+\mathbb{E}_{x}\left(\int_{0}^{\infty}q \mathrm{e}^{-qt}h(X(t))\mathbf{1}_{\{X(t)=\overline{X}(t),\,t<\tau_{b}^{+}\wedge \tau_{\xi}\}}\mathrm{d}t\right)
+\mathbb{E}_{x}\left(h(U(e_{q}))\mathbf{1}_{\{U(e_{q})=\overline{U}(e_{q}),\tau_{\xi}<e_{q}<\kappa_{b}^{+}\}}\right)
\nonumber\\
\hspace{-0.3cm}&:=&\hspace{-0.3cm}\,
qg_{1}(x)+qg_{2}(x)+qg_{3}(x)+qg_{4}(x).\nonumber
\end{eqnarray}


Note that $\int_{0}^{t}\mathbf{1}_{\{X(s)=\overline{X}(s)\}}\mathrm{d}s=W^{(q)}(0) \,\overline{X}(t)$ under $\mathbb{P}_{0}$, see Chapters IV and VII of Bertoin (1996), the proof of Part (ii) of Theorem 1 in Pistorius (2004) or the first three paragraphs in Section 5 of Li et al. (2019).
By \eqref{part1} we have
\begin{eqnarray}\label{qg3}
qg_{3}(x)\hspace{-0.3cm}&=&\hspace{-0.3cm}\mathbb{E}_{x}\left(\int_{0}^{\infty}q \mathrm{e}^{-qL^{-1}(L(t))}h(X(L^{-1}(L(t))))\mathbf{1}_{\{X(t)=\overline{X}(t),\,L^{-1}(L(t))<\tau_{b}^{+}\wedge \tau_{\xi}\}}\mathrm{d}t\right)
\nonumber\\
\hspace{-0.3cm}&=&\hspace{-0.3cm}
W(0)\mathbb{E}_{x}\left(\int_{0}^{\infty}q \mathrm{e}^{-qL^{-1}(L(t))}h(X(L^{-1}(L(t))))\mathbf{1}_{\{L^{-1}(L(t))< \tau_{b}^{+}\wedge \tau_{\xi}\}}\mathrm{d}L_{t}\right)
\nonumber\\
\hspace{-0.3cm}&=&\hspace{-0.3cm}
q W(0)\int_{0}^{b-x} \mathbb{E}_{x}\left(\mathrm{e}^{-qL^{-1}(t)}\mathbf{1}_{\{L^{-1}(t)< \tau_{\xi}\}}\right)h(x+t)\mathrm{d}t
\nonumber\\
\hspace{-0.3cm}&=&\hspace{-0.3cm}
q W(0)\int_{x}^{b} \exp\left(-\int_{x}^{s}\frac{W^{(q)\prime}(\overline{\xi}\left(z\right))}
{W^{(q)}(\overline{\xi}\left(z\right))}\mathrm{d}z\right)h(s)\mathrm{d}s,
\end{eqnarray}
where we have used the fact that $L^{-1}(t)$ has the same law as the first exit time $\tau_{x+t}^{+}$ under $\mathbb{P}_{x}$.

By the strong Markov property of $(U,\overline{U})$, the memoryless property of the exponentially distributed random variable, \eqref{two.sid.exit.Y} and \eqref{12}, we have
\begin{eqnarray}
\label{25}
\hspace{-0.2cm}
qg_{4}(x)\hspace{-0.3cm}&=&\hspace{-0.3cm}
\mathbb{E}_{x}\left(\left.\mathbb{E}_{x}\left(h(U(e_{q}))\mathbf{1}_{\{U(e_{q})=\overline{U}(e_{q}),\tau_{\xi}<e_{q}<\kappa_{b}^{+}\}}\right|\mathcal{F}_{T_{1}}\right)\right)
\nonumber\\
\hspace{-0.3cm}&=&\hspace{-0.3cm}
\mathbb{E}_{x}\left(\mathbf{1}_{\{T_{1}<e_{q}\wedge \kappa_{b}^{+}\}}
\mathbb{E}_{\overline{X}(\tau_{\xi})}\left(h(U(e_{q}))\mathbf{1}_{\{U(e_{q})=\overline{U}(e_{q}),e_{q}<\kappa_{b}^{+}\}}\right)\right)
\nonumber\\
\hspace{-0.3cm}&=&\hspace{-0.3cm}
\mathbb{E}_{x}\left(\mathbf{1}_{\{T_{1}<e_{q}\wedge \kappa_{b}^{+}\}}\left(qg_{3}(\overline{X}(\tau_{\xi}))+qg_{4}(\overline{X}(\tau_{\xi}))\right)\right)
\nonumber\\
\hspace{-0.3cm}&=&\hspace{-0.3cm}
\mathbb{E}_{x}\left(\mathbb{E}_{x}\left(\left.\mathrm{e}^{-qT_{1}}\mathbf{1}_{\{T_{1}< \kappa_{b}^{+}\}}\left(qg_{3}(\overline{X}(\tau_{\xi}))+qg_{4}(\overline{X}(\tau_{\xi}))\right)\right|\mathcal{F}_{\tau_{\xi}}\right)\right)
\nonumber\\
\hspace{-0.3cm}&=&\hspace{-0.3cm}
q\mathbb{E}_{x}\left(\mathrm{e}^{-q\tau_{\xi}}\mathbf{1}_{\{\tau_{\xi}<\kappa_{b}^{+}\}}\frac{1}{Z^{(q)}(\overline{\xi}(\overline{X}(\tau_{\xi})))}\left(g_{3}(\overline{X}(\tau_{\xi}))+g_{4}(\overline{X}(\tau_{\xi}))\right)\right)
\nonumber\\
\hspace{-0.3cm}&=&\hspace{-0.3cm}
q\int_{x}^{b}\left(g_{3}(s)+g_{4}(s)\right)
\exp\left(-\int_{x}^{s}\frac{W^{(q)\prime}(\overline{\xi}(z))}{ W^{(q)}(\overline{\xi}(z))}\mathrm{d}z\right)
\left(\frac{W^{(q)\prime}(\overline{\xi}(s))}{W^{(q)}(\overline{\xi}(s))}-\frac{qW^{(q)}(\overline{\xi}(s))}{Z^{(q)}(\overline{\xi}(s))}\right)\mathrm{d}s,
\end{eqnarray}
where we also used the fact that
$\tau_{\xi}<\kappa_{b}^{+}$ implies $T_{1}<\kappa_{b}^{+}$
(see also \eqref{part2}), and the tact that
$\tau_{\xi}<e_{q}$
combined with
$U(e_{q})=\overline{U}(e_{q})$ implies $T_{1}\leq e_{q}$.

By the compensation formula, the memoryless property for exponential random variable and \eqref{upper.lower.boun.resu.}, $qg_{1}(x)$ can be expressed as
\begin{eqnarray}\label{qg1.raw.}
\hspace{-0.3cm}&&\hspace{-0.3cm}
\mathbb{E}_{x}\left(\int_{0}^{\infty}\sum_{g}\mathrm{e}^{-qg}\prod\limits_{r<g}\mathbf{1}_{\{\overline{\varepsilon}_{r}\leq \overline{\xi}(x+L(r)),\,L(g)\leq b-x\}}\,h\left(x+L(g)-\varepsilon_{g}(t-g)\right)
\right.
\nonumber\\
\hspace{-0.3cm}&&\hspace{0.5cm}
\left.
\times q\mathrm{e}^{-q (t-g)}\mathbf{1}_{\{g<t<g+\zeta_{g}\wedge \rho_{\overline{\xi}(x+L(g))}^{+}(g)\}}\mathrm{d}t\right)
\nonumber\\
\hspace{-0.3cm}&=&\hspace{-0.3cm}
\mathbb{E}_{x}\left(\sum_{g}\mathrm{e}^{-qg}
\prod\limits_{r<g}\mathbf{1}_{\{\overline{\varepsilon}_{r}\leq \overline{\xi}(x+L(r)),\,L(g)\leq b-x\}}\right.
\nonumber\\
\hspace{-0.3cm}&&\hspace{0.5cm}
\left.
\times\int_{0}^{\infty}q \mathrm{e}^{-q s}h\left(x+L(g)-\varepsilon_{g}(s)\right)\mathbf{1}_{\{s<\zeta_{g}\wedge\rho_{\overline{\xi}(x+L(g))}^{+}(g)\}}
\mathrm{d}s\right)
\nonumber\\
\hspace{-0.3cm}&=&\hspace{-0.3cm}
\mathbb{E}_{x}\left(\int_{0}^{\infty}\mathrm{e}^{-qt}\prod\limits_{r<t}\mathbf{1}_{\{\overline{\varepsilon}_{r}\leq \overline{\xi}(x+L(r)),\,L(t)\leq b-x\}}\right.
\nonumber\\
\hspace{-0.3cm}&&\hspace{0.5cm}
\times
\left.\left(\int_{\mathcal{E}}
\int_{0}^{\infty}q \mathrm{e}^{-q s}h\left(x+L(t)-\varepsilon(s)\right)\mathbf{1}_{\{s<\zeta\wedge\rho_{\overline{\xi}(x+L(t))}^{+}\}}
\mathrm{d}s \,n\left(\mathrm{d}\varepsilon\right)\right)\mathrm{d}L(t)\right)
\nonumber\\
\hspace{-0.3cm}&=&\hspace{-0.3cm}
q\int_{0}^{b-x}\mathbb{E}_{x}\left(\mathrm{e}^{-qL_{t-}^{-1}}
\mathbf{1}_{\{L_{t-}^{-1}< \tau_{\xi}\}}\right)
\int_{0}^{\infty}n\left(\mathrm{e}^{-qs}h(x+t-\varepsilon(s))\mathbf{1}_{\{s<\zeta\wedge\rho^{+}_{\overline{\xi}(x+t)}\}}\right)
\mathrm{d}s\mathrm{d}t
\nonumber\\
\hspace{-0.3cm}&=&\hspace{-0.3cm}
q\int_{x}^{b}\exp\left(-\int_{x}^{t}\frac{W^{(q)\prime}(\overline{\xi}(z))}{ W^{(q)}(\overline{\xi}(z))}\mathrm{d}z\right)\int_{0}^{\infty}
n\left(\mathrm{e}^{-qs}h(t-\varepsilon(s))\mathbf{1}_{\{s<\zeta\wedge\rho^{+}_{\overline{\xi}(t)}\}}\right)\mathrm{d}s\mathrm{d}t,
\end{eqnarray}
where $g$ is the left-end point of the excursion $\varepsilon_{g}$, as introduced at the end  of Section 2.
Applying the same arguments as in (\ref{qg3}) and (\ref{qg1.raw.}) we have
\begin{eqnarray}\label{h1}
\hspace{-0.3cm}&&\hspace{-0.3cm}\mathbb{E}_{x}\left(h(X(e_{q}))\mathbf{1}_{\{e_{q}<\tau_{b}^{+}\wedge \tau_{c}^{-}\}}\right)
\nonumber\\
\hspace{-0.3cm}&=&\hspace{-0.3cm}
\mathbb{E}_{x}\left(h(X(e_{q}))\mathbf{1}_{\{X(e_{q})=\overline{X}(e_{q}),\,e_{q}<\tau_{b}^{+}\wedge \tau_{c}^{-}\}}\right)+\mathbb{E}_{x}\left(h(X(e_{q}))\mathbf{1}_{\{X(e_{q})<\overline{X}(e_{q}),\,e_{q}<\tau_{b}^{+}\wedge \tau_{c}^{-}\}}\right)
\nonumber\\
\hspace{-0.3cm}&=&\hspace{-0.3cm}
q\int_{x}^{b}\frac{W^{(q)}(x-c)}{ W^{(q)}(t-c)}\left(W(0)h(t)+\int_{0}^{\infty}n\left(\mathrm{e}^{-qs}h(t-\varepsilon(s))\mathbf{1}_{\{s<\zeta\wedge\rho^{+}_{t-c}\}}\right)\mathrm{d}s\right)\mathrm{d}t,
\end{eqnarray}
where the identity
$$\mathbb{E}_{x-c}\left(\mathrm{e}^{-q \tau_{t-c}^{+}};\tau_{t-c}^{+}<\tau_{0}^{-}\right)=\frac{W^{(q)}(x-c)}{ W^{(q)}(t-c)},\quad -\infty<c\leq x\leq t<\infty,$$
is used.
Equating the right hand sides of (\ref{h1}) and (\ref{h2}) and then differentiating the resulting equation with respect to $b$ gives
\begin{eqnarray}
\hspace{-0.3cm}&&\hspace{-0.3cm}
\frac{W^{(q)}(x-c)}{ W^{(q)}(b-c)}\left(W(0)h(b)+\int_{0}^{\infty}n\left(\mathrm{e}^{-qs}h(b-\varepsilon(s))\mathbf{1}_{\{s<\zeta\wedge\rho^{+}_{b-c}\}}\right)\mathrm{d}s\right)
\nonumber\\
\hspace{-0.3cm}&=&\hspace{-0.3cm}
\frac{W^{(q)}(x-c)}{ W^{(q)}(b-c)}\left(h(b)W(0)+\int_{c}^{b}h(y)\left(W^{(q)\prime}(b-y)-\frac{W^{(q)\prime}(b-c)}{W^{(q)}(b-c)}W^{(q)}(b-y)\right)\mathrm{d}y\right),\nonumber
\end{eqnarray}
or equivalently,
\begin{eqnarray}\label{impo.iden.for.n.01}
\hspace{-0.3cm}&&\hspace{-0.3cm}
\int_{0}^{\infty}n\left(\mathrm{e}^{-qs}h(b-\varepsilon(s))\mathbf{1}_{\{s<\zeta\wedge\rho^{+}_{b-c}\}}\right)\mathrm{d}s
\nonumber\\
\hspace{-0.3cm}&=&\hspace{-0.3cm}
\int_{c}^{b}h(y)\left(W^{(q)\prime}(b-y)-\frac{W^{(q)\prime}(b-c)}{W^{(q)}(b-c)}W^{(q)}(b-y)\right)\mathrm{d}y.
\end{eqnarray}
Combining (\ref{impo.iden.for.n.01}) and (\ref{qg1.raw.}), we get
\begin{eqnarray}\label{qg1}
g_{1}(x)\hspace{-0.3cm}&=&\hspace{-0.3cm}\int_{x}^{b}\mathrm{e}^{-\int_{x}^{s}\frac{W^{(q)\prime}(\overline{\xi}(z))}{ W^{(q)}(\overline{\xi}(z))}\mathrm{d}z}
\int_{\xi(s)}^{s}h(y)\left(W^{(q)\prime}(s-y)-\frac{W^{(q)\prime}(\overline{\xi}(s))}{W^{(q)}(\overline{\xi}(s))}W^{(q)}(s-y)\right)\mathrm{d}y
\mathrm{d}s.
\end{eqnarray}

Using the memoryless property of exponential random variable, $qg_{2}(x)$ can be rewritten as
\begin{eqnarray}\label{qg2}
qg_{2}(x)\hspace{-0.3cm}&=&\hspace{-0.3cm}\mathbb{E}_{x}\left(h(U(e_{q}))\mathbf{1}_{\{U(e_{q})<\overline{U}(e_{q}),\,\tau_{\xi}<e_{q}<T_{1}<\kappa_{b}^{+}\}}\right)
\nonumber\\
\hspace{-0.3cm}&&\hspace{-0.3cm}
+\mathbb{E}_{x}\left(h(U(e_{q}))\mathbf{1}_{\{U(e_{q})<\overline{U}(e_{q}),\,\tau_{\xi}<T_{1}<e_{q}<\kappa_{b}^{+}\}}\right)
\nonumber\\
\hspace{-0.3cm}&=&\hspace{-0.3cm}\mathbb{E}_{x}\left(h(U(e_{q}))\mathbf{1}_{\{\tau_{\xi}<e_{q}<T_{1}<\kappa_{b}^{+}\}}\right)
+\mathbb{E}_{x}\left(h(U(e_{q}))\mathbf{1}_{\{U(e_{q})<\overline{U}(e_{q}),\,\tau_{\xi}<T_{1}<e_{q}<\kappa_{b}^{+}\}}\right)
\nonumber\\
\hspace{-0.3cm}&:=&\hspace{-0.3cm}\,qg_{21}(x)+qg_{22}(x),
\end{eqnarray}
where we also took use of the fact that $\tau_{\xi}<\kappa_{b}^{+}$ implies $T_{1}<\kappa_{b}^{+}$ as in \eqref{part2}.

Using (\ref{reso.meas.Y}) and \eqref{12} 
we have
\begin{eqnarray}\label{qg21}
qg_{21}(x)\hspace{-0.3cm}&=&\hspace{-0.3cm}\mathbb{E}_{x}\left(\mathbb{E}_{x}\left(\left.h(U(e_{q}))\mathbf{1}_{\{\tau_{\xi}<e_{q}<T_{1}<\kappa_{b}^{+}\}}\right|\mathcal{F}_{\tau_{\xi}}\right)\right)
\nonumber\\
\hspace{-0.3cm}&=&\hspace{-0.3cm}
\mathbb{E}_{x}\left(\mathbf{1}_\{\tau_{\xi}<e_{q}\wedge\kappa_{b}^{+}\}\mathbb{E}_{x}\left(\left.h(U(e_{q}))\mathbf{1}_{\{e_{q}<T_{1}\}}\right|\mathcal{F}_{\tau_{\xi}}\right)\right)
\nonumber\\
\hspace{-0.3cm}&=&\hspace{-0.3cm}
\mathbb{E}_{x}\left(\mathbf{1}_{\{\tau_{\xi}<e_{q}\wedge\kappa_{b}^{+}\}}\left.\mathbb{E}_{}\left(h(\xi(z)+Y(e_{q}))\mathbf{1}_{\{e_{q}<\sigma^{+}_{\overline{\xi}(z)}\}}\right)\right|_{z=\overline{X}(\tau_{\xi})}\right)
\nonumber\\
\hspace{-0.3cm}&=&\hspace{-0.3cm}
q\mathbb{E}_{x}\left(\mathbf{1}_{\{\tau_{\xi}<e_{q}\wedge\kappa_{b}^{+}\}}\int_{0}^{\overline{\xi}(\overline{X}(\tau_{\xi}))}h(\xi(\overline{X}(\tau_{\xi}))+y)
\frac{W^{(q)}(\overline{\xi}(\overline{X}(\tau_{\xi}))-y)}{Z^{(q)}(\overline{\xi}(\overline{X}(\tau_{\xi})))}\mathrm{d}y\right)
\nonumber\\
\hspace{-0.3cm}&=&\hspace{-0.3cm}
q\int_{x}^{b}
\exp\left(-\int_{x}^{s}\frac{W^{(q)\prime}(\overline{\xi}(z))}{ W^{(q)}(\overline{\xi}(z))}\mathrm{d}z\right)
\left(\frac{W^{(q)\prime}(\overline{\xi}(s))}{W^{(q)}(\overline{\xi}(s))}-\frac{qW^{(q)}(\overline{\xi}(s))}{Z^{(q)}(\overline{\xi}(s))}\right)
\nonumber\\
\hspace{-0.3cm}&&\hspace{-0.3cm}
\,\,\,\,\,\,\,\,\,\times\int_{0}^{\overline{\xi}(s)}h(\xi(s)+y)
W^{(q)}(\overline{\xi}(s)-y)\mathrm{d}y\mathrm{d}s.
\end{eqnarray}

In addition, observing that $\tau_\xi<T_1$ for $\tau_\xi<\infty$, by (\ref{two.sid.exit.Y}) and \eqref{12} one can rewrite $qg_{22}(x)$ as
\begin{eqnarray}\label{qg22}
qg_{22}(x)\hspace{-0.3cm}&=&\hspace{-0.3cm}
\mathbb{E}_{x}\left(\mathbb{E}_{x}\left(\left.h(U(e_{q}))\mathbf{1}_{\{U(e_{q})<\overline{U}(e_{q}),\,\tau_{\xi}<T_{1}<e_{q}<\kappa_{b}^{+}\}}\right|\mathcal{F}_{T_{1}}\right)\right)
\nonumber\\
\hspace{-0.3cm}&=&\hspace{-0.3cm}
\mathbb{E}_{x}\left(\mathbf{1}_{\{T_{1}<e_{q}\wedge\kappa_{b}^{+}\}}
\left(qg_{1}(\overline{X}(\tau_{\xi}))+qg_{2}(\overline{X}(\tau_{\xi}))\right)\right)
\nonumber\\
\hspace{-0.3cm}&=&\hspace{-0.3cm}
\mathbb{E}_{x}\left(\mathbb{E}_{x}\left(\left.\mathrm{e}^{-qT_{1}}\mathbf{1}_{\{T_{1}<\kappa_{b}^{+}\}}
\left(qg_{1}(\overline{X}(\tau_{\xi}))+qg_{2}(\overline{X}(\tau_{\xi}))\right)\right|\mathcal{F}_{\tau_{\xi}}\right)\right)
\nonumber\\
\hspace{-0.3cm}&=&\hspace{-0.3cm}
\mathbb{E}_{x}\left(\mathrm{e}^{-q\tau_{\xi}}\mathbf{1}_{\{\tau_{\xi}<\kappa_{b}^{+}\}}\frac{1}{Z^{(q)}(\overline{\xi}(\overline{X}(\tau_{\xi})))}\left(qg_{1}(\overline{X}(\tau_{\xi}))+qg_{2}(\overline{X}(\tau_{\xi}))\right)\right)
\nonumber\\
\hspace{-0.3cm}&=&\hspace{-0.3cm}
q\int_{x}^{b}
\left(g_{1}(s)+g_{2}(s)\right)\mathrm{e}^{-\int_{x}^{s}\frac{W^{(q)\prime}(\overline{\xi}(z))}{ W^{(q)}(\overline{\xi}(z))}\mathrm{d}z}
\left(\frac{W^{(q)\prime}(\overline{\xi}(s))}{W^{(q)}(\overline{\xi}(s))}-\frac{qW^{(q)}(\overline{\xi}(s))}{Z^{(q)}(\overline{\xi}(s))}\right)\mathrm{d}s.
\end{eqnarray}

Combining (\ref{qg3}), \eqref{25}, (\ref{qg1}), (\ref{qg2}), (\ref{qg21}) and (\ref{qg22}), we obtain the following differential equation on $g(x)$.
\begin{eqnarray}\label{g'}
g^{\prime}(x)\hspace{-0.3cm}&=&\hspace{-0.3cm}\frac{W^{(q)\prime}(\overline{\xi}(x))}{ W^{(q)}(\overline{\xi}(x))}g(x)- W(0)h(x)
-(g_{3}(x)+g_{4}(x))
\left(\frac{W^{(q)\prime}(\overline{\xi}(x))}{W^{(q)}(\overline{\xi}(x))}-\frac{qW^{(q)}(\overline{\xi}(s))}{Z^{(q)}(\overline{\xi}(s))}\right)
\nonumber\\
\hspace{-0.3cm}&&\hspace{-0.3cm}
-\int_{\xi(x)}^{x}h(y)\left(W^{(q)\prime}(x-y)-\frac{W^{(q)\prime}(\overline{\xi}(x))}{W^{(q)}(\overline{\xi}(x))}W^{(q)}(x-y)\right)\mathrm{d}y
\nonumber\\
\hspace{-0.3cm}&&\hspace{-0.3cm}
-\left(g_{1}(x)+g_{2}(x)\right)
\left(\frac{W^{(q)\prime}(\overline{\xi}(x))}{W^{(q)}(\overline{\xi}(x))}-\frac{qW^{(q)}(\overline{\xi}(x))}{Z^{(q)}(\overline{\xi}(x))}\right)
\nonumber\\
\hspace{-0.3cm}&&\hspace{-0.3cm}
-\left(\frac{W^{(q)\prime}(\overline{\xi}(x))}{W^{(q)}(\overline{\xi}(x))}-\frac{qW^{(q)}(\overline{\xi}(x))}{Z^{(q)}(\overline{\xi}(x))}\right)
\int_{0}^{\overline{\xi}(x)}h(\xi(x)+y)
W^{(q)}(\overline{\xi}(x)-y)\mathrm{d}y
\nonumber\\
\hspace{-0.3cm}&=&\hspace{-0.3cm}
\frac{qW^{(q)}(\overline{\xi}(x))}{Z^{(q)}(\overline{\xi}(x))}g(x)- W(0)h(x)
-\int_{\xi(x)}^{x}h(y)\left(W^{(q)\prime}(x-y)-\frac{W^{(q)\prime}(\overline{\xi}(x))}{W^{(q)}(\overline{\xi}(x))}W^{(q)}(x-y)\right)\mathrm{d}y
\nonumber\\
\hspace{-0.3cm}&&\hspace{-0.3cm}
-\left(\frac{W^{(q)\prime}(\overline{\xi}(x))}{W^{(q)}(\overline{\xi}(x))}-\frac{qW^{(q)}(\overline{\xi}(x))}{Z^{(q)}(\overline{\xi}(x))}\right)
\int_{\xi(x)}^{x}h(y)
W^{(q)}(x-y)\mathrm{d}y
\nonumber\\
\hspace{-0.3cm}&=&\hspace{-0.3cm}
\frac{qW^{(q)}(\overline{\xi}(x))}{Z^{(q)}(\overline{\xi}(x))}g(x)- W(0)h(x)
-\int_{\xi(x)}^{x}h(y)\left(W^{(q)\prime}(x-y)-\frac{qW^{(q)}(\overline{\xi}(x))}{W^{(q)}(\overline{\xi}(x))}W^{(q)}(x-y)\right)\mathrm{d}y,
\end{eqnarray}
with boundary condition $g(b)=0$. Solving equation (\ref{g'}) yields
\begin{eqnarray}
\label{eq.gene.reso.meas.}
g(x)\hspace{-0.3cm}&=&\hspace{-0.3cm}\int_{0}^{\infty}\mathrm{e}^{-qt}\mathbb{E}_{x}(h(U(t));t<\kappa_{b}^{+})\mathrm{d}t
\nonumber\\
\hspace{-0.3cm}&=&\hspace{-0.3cm}
W^{(q)}(0)\int_{x}^{b}h(y)
\exp\left(-\int_{x}^{y}\frac{qW^{(q)}(\overline{\xi}(z))}{ Z^{(q)}(\overline{\xi}(z))}\mathrm{d}z\right)\mathrm{d}y
+\int_{x}^{b}\exp\left(-\int_{x}^{y}\frac{qW^{(q)}(\overline{\xi}(z))}{ Z^{(q)}(\overline{\xi}(z))}\mathrm{d}z\right)
\nonumber\\
\hspace{-0.3cm}&&\hspace{-0.3cm}\,\,\,\,\,\,\,\,\,
\times\int_{\xi(y)}^{y}h(z)\left(W^{(q)\prime}(y-z)-\frac{qW^{(q)}(\overline{\xi}(y))}{Z^{(q)}(\overline{\xi}(y))}W^{(q)}(y-z)\right)\mathrm{d}z  \mathrm{d}y.
\end{eqnarray}
The resolvent density (\ref{resovent.meas.}) follows immediately from (\ref{eq.gene.reso.meas.}).
\end{proof}

\vspace{0.2cm}
\begin{rem}
Note that
\begin{eqnarray}
\label{eq.gene.reso.meas.1}
\hspace{-0.3cm}&&\hspace{-0.3cm}\int_{0}^{\infty}\mathrm{e}^{-qt}\mathbb{P}_{x}(U(t)\leq u,t<\kappa_{b}^{+})\mathrm{d}t
\nonumber\\
\hspace{-0.3cm}&=&\hspace{-0.3cm}
W^{(q)}(0)\left(\int_{x}^{u}
\exp\left(-\int_{x}^{y}\frac{qW^{(q)}(\overline{\xi}(z))}{ Z^{(q)}(\overline{\xi}(z))}\mathrm{d}z\right)\mathrm{d}y\,\mathbf{1}_{(x,b)}(u)
+\int_{x}^{b}
\exp\left(-\int_{x}^{y}\frac{qW^{(q)}(\overline{\xi}(z))}{ Z^{(q)}(\overline{\xi}(z))}\mathrm{d}z\right)\mathrm{d}y\,\mathbf{1}_{u\geq b}\right)
\nonumber\\
\hspace{-0.3cm}&&\hspace{-0.3cm}
+\int_{x}^{b}\exp\left(-\int_{x}^{y}\frac{qW^{(q)}(\overline{\xi}(z))}{ Z^{(q)}(\overline{\xi}(z))}\mathrm{d}z\right)
\left(\int_{\xi(y)}^{y}
\left(W^{(q)\prime}(y-z)-\frac{qW^{(q)}(\overline{\xi}(y))}{Z^{(q)}(\overline{\xi}(y))}W^{(q)}(y-z)\right)\mathrm{d}z\,\mathbf{1}_{u\geq y}\right.
\nonumber\\
\hspace{-0.3cm}&&\hspace{-0.3cm}\,\,\,\,\,\,\,\,\,
\left.\,\,\,\,\,\,\,\,\,\,\,\,
+\int_{\xi(y)}^{u}\left(W^{(q)\prime}(y-z)-\frac{qW^{(q)}(\overline{\xi}(y))}{Z^{(q)}(\overline{\xi}(y))}W^{(q)}(y-z)\right)\mathrm{d}z\,\mathbf{1}_{\xi(y)<u< y}\right)  \mathrm{d}y.
\end{eqnarray}
Letting $u=b$ in (\ref{eq.gene.reso.meas.1}), we get
\begin{eqnarray}
\label{}
\hspace{-0.3cm}&&\hspace{-0.3cm}\frac{1-\mathbb{E}_{x}(\mathrm{e}^{-q\kappa_{b}^{+}})}{q}
=\int_{0}^{\infty}\mathrm{e}^{-qt}\mathbb{P}_{x}(t<\kappa_{b}^{+})\mathrm{d}t
\nonumber\\
\hspace{-0.3cm}&=&\hspace{-0.3cm}
\int_{0}^{\infty}\mathrm{e}^{-qt}\mathbb{P}_{x}(U(t)\leq b,t<\kappa_{b}^{+})\mathrm{d}t
\nonumber\\
\hspace{-0.3cm}&=&\hspace{-0.3cm}
W^{(q)}(0)\int_{x}^{b}
\exp\left(-\int_{x}^{y}\frac{qW^{(q)}(\overline{\xi}(z))}{ Z^{(q)}(\overline{\xi}(z))}\mathrm{d}z\right)\mathrm{d}y
\nonumber\\
\hspace{-0.3cm}&&\hspace{-0.3cm}
+\int_{x}^{b}\exp\left(-\int_{x}^{y}\frac{qW^{(q)}(\overline{\xi}(z))}{ Z^{(q)}(\overline{\xi}(z))}\mathrm{d}z\right)
\int_{\xi(y)}^{y}\left(W^{(q)\prime}(y-z)-\frac{qW^{(q)}(\overline{\xi}(y))}{Z^{(q)}(\overline{\xi}(y))}W^{(q)}(y-z)\right)\mathrm{d}z  \mathrm{d}y
\nonumber\\
\hspace{-0.3cm}&=&\hspace{-0.3cm}
\int_{x}^{b}\exp\left(-\int_{x}^{y}\frac{qW^{(q)}(\overline{\xi}(z))}{ Z^{(q)}(\overline{\xi}(z))}\mathrm{d}z\right)
\frac{1}{q}\frac{qW^{(q)}(\overline{\xi}(y))}{Z^{(q)}(\overline{\xi}(y))} \mathrm{d}y
\nonumber\\
\hspace{-0.3cm}&=&\hspace{-0.3cm}
\frac{1}{q}-\frac{1}{q}\exp\left(-\int_{x}^{b}\frac{qW^{(q)}(\overline{\xi}(z))}{ Z^{(q)}(\overline{\xi}(z))}\mathrm{d}z\right),\nonumber
\end{eqnarray}
which coincides with (\ref{upper.lower.boun.resu.}) of Proposition \ref{Laplace.tra.ucro.}.
\end{rem}

\vspace{0.2cm}
\begin{rem}
Letting $\xi\equiv0$ in (\ref{resovent.meas.}) and noting that
\[\exp\left(-\int_{x}^{u}\frac{qW^{(q)}(z)}{ Z^{(q)}(z)}\mathrm{d}z\right)=\frac{Z^{(q)}(x)}{Z^{(q)}(u)},\]
 we get
\begin{eqnarray}
\label{}
\hspace{-0.3cm}&&\hspace{-0.3cm}\int_{0}^{\infty}\mathrm{e}^{-qt}\mathbb{P}_{x}(U(t)\in \mathrm{d}u,t<\kappa_{b}^{+})\mathrm{d}t
\nonumber\\
\hspace{-0.3cm}&=&\hspace{-0.3cm}
\frac{W^{(q)}(0)Z^{(q)}(x)\mathbf{1}_{(x,b)}(u)}{Z^{(q)}(u)}\mathrm{d}u
+\int_{x}^{b}\frac{Z^{(q)}(x)}{Z^{(q)}(y)}\left(W^{(q)\prime}(y-u)-\frac{qW^{(q)}(y)}{Z^{(q)}(y)}W^{(q)}(y-u)\right)\mathbf{1}_{(0,y)}(u)\mathrm{d}y  \mathrm{d}u
\nonumber\\
\hspace{-0.3cm}&=&\hspace{-0.3cm}
\frac{W^{(q)}(0)Z^{(q)}(x)\mathbf{1}_{(x,b)}(u)}{Z^{(q)}(u)}\mathrm{d}u
+Z^{(q)}(x)\int_{x}^{b}\frac{\mathrm{d}}{\mathrm{d}y}\left[\frac{W^{(q)}(y-u)}{Z^{(q)}(y)}\right]\mathbf{1}_{(0,y)}(u)\mathrm{d}y  \mathrm{d}u
\nonumber\\
\hspace{-0.3cm}&=&\hspace{-0.3cm}
\left(\frac{W^{(q)}(0)}{Z^{(q)}(u)}
+\left.\frac{W^{(q)}(y-u)}{Z^{(q)}(y)}\right|_{u}^{b}\right) Z^{(q)}(x)\mathbf{1}_{(x,b)}(u) \mathrm{d}u
+
\left.\frac{W^{(q)}(y-u)}{Z^{(q)}(y)}\right|_{x}^{b} Z^{(q)}(x)\mathbf{1}_{(0,x)}(u) \mathrm{d}u
\nonumber\\
\hspace{-0.3cm}&=&\hspace{-0.3cm}
\left(\frac{Z^{(q)}(x)}{Z^{(q)}(b)} W^{(q)}(b-u)-W^{(q)}(x-u)\right) \mathbf{1}_{(0,b)}(u) \mathrm{d}u
,\nonumber
\end{eqnarray}
which coincides with (i) of Theorem 1 in Pistorius (2004).
\end{rem}

\vspace{0.2cm}
The following result gives an expression of the expectation of the total discounted capital injections until time $\kappa_{b}^{+}$.

\vspace{0.2cm}
\begin{thm}\label{3.3}
For $q>0$ and $x\leq b$, we have
\begin{eqnarray}\label{expr.of.exp.tot.dis.cap.inj.}
V_\xi(x;b)
\hspace{-0.3cm}&=&\hspace{-0.3cm}
\int_{x}^{b}\left(Z^{(q)}(\overline{\xi}(y))-\psi^{\prime}(0+)W^{(q)}(\overline{\xi}(y))\right)
\exp\left(-\int_{x}^{y}\frac{qW^{(q)}(\overline{\xi}(z))}{ Z^{(q)}(\overline{\xi}(z))}\mathrm{d}z\right)\mathrm{d}y
\nonumber\\
\hspace{-0.3cm}&&\hspace{-0.3cm}
+\int_{x}^{b}\exp\left(-\int_{x}^{y}\frac{qW^{(q)}(\overline{\xi}(z))}{ Z^{(q)}(\overline{\xi}(z))}\mathrm{d}z\right)
\frac{qW^{(q)}(\overline{\xi}(y))}{Z^{(q)}(\overline{\xi}(y))}\left(-\overline{Z}^{(q)}(\overline{\xi}(y))+\psi^{\prime}(0+)\overline{W}^{(q)}(\overline{\xi}(y))\right)  \mathrm{d}y.
\end{eqnarray}
\end{thm}

\begin{proof}[Proof:]\,\,\,
For $q>0$ and $x\leq b$, we have
\begin{eqnarray}\label{disc.total.cap.}
V_\xi(x;b)\hspace{-0.3cm}&=&\hspace{-0.3cm}\mathbb{E}_{x}\left(\int_{0}^{\kappa_{b}^{+}}\mathrm{e}^{-q t}\mathrm{d}R(t)\right)
=\mathbb{E}_{x}\left(\mathrm{e}^{-q\tau_{\xi}}\mathbf{1}_{\{\tau_{\xi}<\kappa_{b}^{+}\}}\left(\xi(\overline{X}(\tau_{\xi}))-X(\tau_{\xi})\right)\right)
\nonumber\\
\hspace{-0.3cm}&&\hspace{-0.3cm}
+\mathbb{E}_{x}\left(\mathbf{1}_{\{\tau_{\xi}<\kappa_{b}^{+}\}}\int_{\tau_{\xi}+}^{T_{1}}\mathrm{e}^{-q t}\mathrm{d}R(t)\right)
+\mathbb{E}_{x}\left(\mathbf{1}_{\{\tau_{\xi}<\kappa_{b}^{+}\}}\int_{T_{1}}^{\kappa_{b}^{+}}\mathrm{e}^{-q t}\mathrm{d}R(t)\right)
\nonumber\\
\hspace{-0.3cm}&:=&\hspace{-0.3cm}V_{1}(x;b)+V_{2}(x;b)+V_{3}(x;b).
\end{eqnarray}

By \eqref{14}, $V_{1}(x;b)$ can be expressed as
\begin{eqnarray}\label{l1}
V_{1}(x;b)
\hspace{-0.3cm}&=&\hspace{-0.3cm}\int_{x}^{b}
\exp\left(-\int_{x}^{s}\frac{W^{(q)\prime}(\overline{\xi}\left(z\right))}
{W^{(q)}(\overline{\xi}\left(z\right))}\mathrm{d}z\right)
\left(Z^{(q)}(\overline{\xi}(s))-\psi^{\prime}(0+)W^{(q)}(\overline{\xi}(s))\right.
\nonumber\\
\hspace{-0.3cm}&&\hspace{-0.3cm}\hspace{4cm}
\left.-\frac{\overline{Z}^{(q)}(\overline{\xi}(s))-\psi^{\prime}(0+)\overline{W}^{(q)}(\overline{\xi}(s))}{W^{(q)}(\overline{\xi}(s))}W^{(q)\prime}(\overline{\xi}(s))\right)\mathrm{d}s.
\end{eqnarray}
By the Markov property for the reflected process $(U,\overline{U})$,
\begin{eqnarray}
\label{v0(0}
V_{2}(x;b)\hspace{-0.3cm}&=&\hspace{-0.3cm}
\mathbb{E}_{x}\left(\mathrm{e}^{-q\tau_{\xi}}\mathbf{1}_{\{\tau_{\xi}<\tau_{b}^{+}\}}
V_0(0;\overline{\xi}(\overline{X}(\tau_{\xi})))
\right).
\end{eqnarray}
By the proof of Theorem 1 of Avram et al. (2007),  we have
$$V_0(0;\overline{\xi}(\overline{X}(\tau_{\xi})))=-\frac{\psi^{\prime}(0+)}{q}+\frac{\overline{Z}^{(q)}(\overline{\xi}(\overline{X}(\tau_{\xi})))+\frac{\psi^{\prime}(0+)}{q}}{Z^{(q)}(\overline{\xi}(\overline{X}(\tau_{\xi})))},$$
which together with \eqref{v0(0} and \eqref{12} gives
\begin{eqnarray}
\label{l2.}
V_{2}(x;b)\hspace{-0.3cm}&=&\hspace{-0.3cm}
\mathbb{E}_{x}\left(\mathrm{e}^{-q\tau_{\xi}}\mathbf{1}_{\{\tau_{\xi}<\tau_{b}^{+}\}}\left(-\frac{\psi^{\prime}(0+)}{q}+\frac{\overline{Z}^{(q)}(\overline{\xi}(\overline{X}(\tau_{\xi})))+\frac{\psi^{\prime}(0+)}{q}}{Z^{(q)}(\overline{\xi}(\overline{X}(\tau_{\xi})))}\right)\right)
\nonumber\\
\hspace{-0.3cm}&=&\hspace{-0.3cm}
\int_{x}^{b}\left(-\frac{\psi^{\prime}(0+)}{q}+\frac{\overline{Z}^{(q)}(\overline{\xi}(s))+\frac{\psi^{\prime}(0+)}{q}}{Z^{(q)}(\overline{\xi}(s))}\right)
\exp\left(-\int_{x}^{s}\frac{W^{(q)\prime}(\overline{\xi}(z))}{ W^{(q)}(\overline{\xi}(z))}\mathrm{d}z\right)
\nonumber\\
\hspace{-0.3cm}&&\hspace{-0.3cm}\,\,\,\,\,\,\,\times
\left(\frac{W^{(q)\prime}(\overline{\xi}(s))}{W^{(q)}(\overline{\xi}(s))}Z^{(q)}(\overline{\xi}(s))-qW^{(q)}(\overline{\xi}(s))\right)\mathrm{d}s.
\end{eqnarray}

Making use of \eqref{12}  again, one can get
\begin{eqnarray}\label{l3}
V_{3}(x;b)\hspace{-0.3cm}&=&\hspace{-0.3cm}
\mathbb{E}_{x}\left(\mathbb{E}_{x}\left(\left.\mathbf{1}_{\{\tau_{\xi}<\kappa_{b}^{+}\}}\int_{T_{1}}^{\kappa_{b}^{+}}\mathrm{e}^{-q t}\mathrm{d}R(t)\right|\mathcal{F}_{T_{1}}\right)\right)
\nonumber\\
\hspace{-0.3cm}&=&\hspace{-0.3cm}
\mathbb{E}_{x}\left(\mathbb{E}_{x}\left(\left.\mathrm{e}^{-qT_{1}}\mathbf{1}_{\{\tau_{\xi}<\kappa_{b}^{+}\}}V_\xi(\overline{X}(\tau_{\xi});b)\right|\mathcal{F}_{\tau_{\xi}}\right)\right)
\nonumber\\
\hspace{-0.3cm}&=&\hspace{-0.3cm}
\mathbb{E}_{x}\left(\mathrm{e}^{-q\tau_{\xi}}\mathbf{1}_{\{\tau_{\xi}<\tau_{b}^{+}\}}\frac{1}{Z^{(q)}(\overline{\xi}(\overline{X}(\tau_{\xi})))}V_\xi(\overline{X}(\tau_{\xi});b)\right)
\nonumber\\
\hspace{-0.3cm}&=&\hspace{-0.3cm}
\int_{x}^{b}V_\xi(s;b)
\mathrm{e}^{-\int_{x}^{s}\frac{W^{(q)\prime}(\overline{\xi}(z))}{ W^{(q)}(\overline{\xi}(z))}\mathrm{d}z}
\left(\frac{W^{(q)\prime}(\overline{\xi}(s))}{W^{(q)}(\overline{\xi}(s))}-\frac{qW^{(q)}(\overline{\xi}(s))}{Z^{(q)}(\overline{\xi}(s))}\right)\mathrm{d}s.
\end{eqnarray}

Denote by $V_{\xi}^{\prime}(x;b)$ the derivative of $V_{\xi}(x;b)$ with respect to its first argument. Combining (\ref{disc.total.cap.}), (\ref{l1}), (\ref{l2.}) and (\ref{l3}) we have
\begin{eqnarray}\label{l'}
V_{\xi}^{\prime}(x;b)\hspace{-0.3cm}&=&\hspace{-0.3cm}\frac{W^{(q)\prime}(\overline{\xi}(x))}{ W^{(q)}(\overline{\xi}(x))}V_\xi(x;b)
-V_\xi(x;b)
\left(\frac{W^{(q)\prime}(\overline{\xi}(x))}{W^{(q)}(\overline{\xi}(x))}-\frac{qW^{(q)}(\overline{\xi}(x))}{Z^{(q)}(\overline{\xi}(x))}\right)
\nonumber\\
\hspace{-0.3cm}&&\hspace{-0.3cm}
-\left(Z^{(q)}(\overline{\xi}(x))-\psi^{\prime}(0+)W^{(q)}(\overline{\xi}(x))-\frac{\overline{Z}^{(q)}(\overline{\xi}(x))-\psi^{\prime}(0+)\overline{W}^{(q)}(\overline{\xi}(x))}{W^{(q)}(\overline{\xi}(x))}W^{(q)\prime}(\overline{\xi}(x))\right)
\nonumber\\
\hspace{-0.3cm}&&\hspace{-0.3cm}
-\left(-\frac{\psi^{\prime}(0+)}{q}+\frac{\overline{Z}^{(q)}(\overline{\xi}(x))+\frac{\psi^{\prime}(0+)}{q}}{Z^{(q)}(\overline{\xi}(x))}\right)\left(\frac{W^{(q)\prime}(\overline{\xi}(x))}{W^{(q)}(\overline{\xi}(x))}Z^{(q)}(\overline{\xi}(x))-qW^{(q)}(\overline{\xi}(x))\right)
\nonumber\\
\hspace{-0.3cm}&=&\hspace{-0.3cm}
\frac{qW^{(q)}(\overline{\xi}(x))}{Z^{(q)}(\overline{\xi}(x))}V_\xi(x;b)
\nonumber\\
\hspace{-0.3cm}&&\hspace{-0.3cm}
-Z^{(q)}(\overline{\xi}(x))+\psi^{\prime}(0+)W^{(q)}(\overline{\xi}(x))+\frac{\overline{Z}^{(q)}(\overline{\xi}(x))-\psi^{\prime}(0+)\overline{W}^{(q)}(\overline{\xi}(x))}{W^{(q)}(\overline{\xi}(x))}W^{(q)\prime}(\overline{\xi}(x))
\nonumber\\
\hspace{-0.3cm}&&\hspace{-0.3cm}
+\frac{-\overline{Z}^{(q)}(\overline{\xi}(x))+\psi^{\prime}(0+)\overline{W}^{(q)}(\overline{\xi}(x))}{Z^{(q)}(\overline{\xi}(x))}\left(\frac{W^{(q)\prime}(\overline{\xi}(x))}{W^{(q)}(\overline{\xi}(x))}Z^{(q)}(\overline{\xi}(x))-qW^{(q)}(\overline{\xi}(x))\right)
\nonumber\\
\hspace{-0.3cm}&=&\hspace{-0.3cm}
\frac{qW^{(q)}(\overline{\xi}(x))}{Z^{(q)}(\overline{\xi}(x))}V_\xi(x;b)
-Z^{(q)}(\overline{\xi}(x))+\psi^{\prime}(0+)W^{(q)}(\overline{\xi}(x))
\nonumber\\
\hspace{-0.3cm}&&\hspace{-0.3cm}
-\frac{qW^{(q)}(\overline{\xi}(x))}{Z^{(q)}(\overline{\xi}(x))}\left(-\overline{Z}^{(q)}(\overline{\xi}(x))+\psi^{\prime}(0+)\overline{W}^{(q)}(\overline{\xi}(x))\right)
.
\end{eqnarray}
Solving (\ref{l'}) with boundary condition $V_\xi(b, b)=0 $,  we obtain (\ref{expr.of.exp.tot.dis.cap.inj.}).
\end{proof}

\vspace{0.2cm}
\begin{rem}
In particular, for $\xi\equiv0$ we have
\begin{eqnarray}
V_0(x;b)\hspace{-0.3cm}&
=&\hspace{-0.3cm}\int_{x}^{b}\left(Z^{(q)}(y)-\psi^{\prime}(0+)W^{(q)}(y)\right)
\exp\left(-\int_{x}^{y}\frac{qW^{(q)}(z)}{ Z^{(q)}(z)}\mathrm{d}z\right)\mathrm{d}y
\nonumber\\
\hspace{-0.3cm}&&\hspace{-0.3cm}
+\int_{x}^{b}\exp\left(-\int_{x}^{y}\frac{qW^{(q)}(z)}{ Z^{(q)}(z)}\mathrm{d}z\right)
\frac{qW^{(q)}(y)}{Z^{(q)}(y)}\left(-\overline{Z}^{(q)}(y)+\psi^{\prime}(0+)\overline{W}^{(q)}(y)\right)  \mathrm{d}y
\nonumber\\
\hspace{-0.3cm}&
=&\hspace{-0.3cm}\int_{x}^{b}\left(Z^{(q)}(y)-\psi^{\prime}(0+)W^{(q)}(y)\right)\frac{Z^{(q)}(x)}{ Z^{(q)}(y)}\mathrm{d}y
\nonumber\\
\hspace{-0.3cm}&&\hspace{-0.3cm}
+\int_{x}^{b}\frac{Z^{(q)}(x)}{ Z^{(q)}(y)}
\frac{qW^{(q)}(y)}{Z^{(q)}(y)}\left(-\overline{Z}^{(q)}(y)+\psi^{\prime}(0+)\overline{W}^{(q)}(y)\right)  \mathrm{d}y
\nonumber\\
\hspace{-0.3cm}&=&\hspace{-0.3cm}
-Z^{(q)}(x)\left.\frac{-\overline{Z}^{(q)}(y)+\psi^{\prime}(0+)\overline{W}^{(q)}(y)}{ Z^{(q)}(y)}\right|_{y=x}^{y=b}
\nonumber\\
\hspace{-0.3cm}&=&\hspace{-0.3cm}
-\overline{Z}^{(q)}(x)+\psi^{\prime}(0+)\overline{W}^{(q)}(x)+\frac{Z^{(q)}(x)}{ Z^{(q)}(b)}\left(\overline{Z}^{(q)}(b)-\psi^{\prime}(0+)\overline{W}^{(q)}(b)\right)
\nonumber\\
\hspace{-0.3cm}&=&\hspace{-0.3cm}
-\overline{Z}^{(q)}(x)-\frac{\psi^{\prime}(0+)}{q}+\frac{Z^{(q)}(x)}{ Z^{(q)}(b)}\left(\overline{Z}^{(q)}(b)+\frac{\psi^{\prime}(0+)}{q}\right),\nonumber
\end{eqnarray}
which coincides with the corresponding results (the first block of equations) on page 167 of Avram et al. (2007).
\end{rem}

\vspace{0.2cm}
\begin{rem}
Letting $b\rightarrow\infty$ in (\ref{expr.of.exp.tot.dis.cap.inj.}), we recover the following expression of the expected total discounted capital injections.
\begin{eqnarray}\label{expr.of.exp.tot.dis.cap.inj.v11}
V_\xi(x;\infty)\hspace{-0.3cm}&
=&\hspace{-0.3cm}\int_{x}^{\infty}\left(Z^{(q)}(\overline{\xi}(y))-\psi^{\prime}(0+)W^{(q)}(\overline{\xi}(y))\right)
\exp\left(-\int_{x}^{y}\frac{qW^{(q)}(\overline{\xi}(z))}{ Z^{(q)}(\overline{\xi}(z))}\mathrm{d}z\right)\mathrm{d}y
\nonumber\\
\hspace{-0.3cm}&&\hspace{-0.3cm}
+\int_{x}^{\infty}\exp\left(-\int_{x}^{y}\frac{qW^{(q)}(\overline{\xi}(z))}{ Z^{(q)}(\overline{\xi}(z))}\mathrm{d}z\right)
\frac{qW^{(q)}(\overline{\xi}(y))}{Z^{(q)}(\overline{\xi}(y))}\left(-\overline{Z}^{(q)}(\overline{\xi}(y))+\psi^{\prime}(0+)\overline{W}^{(q)}(\overline{\xi}(y))\right)  \mathrm{d}y.\nonumber
\end{eqnarray}
Letting $\xi(x)\equiv0$ in the above equality, we get
\begin{eqnarray}\label{}
V_0(x;\infty)
\hspace{-0.3cm}&
=&\hspace{-0.3cm}
-\overline{Z}^{(q)}(x)-\frac{\psi^{\prime}(0+)}{q}+\lim\limits_{b\rightarrow\infty}\frac{Z^{(q)}(x)}{ Z^{(q)}(b)}\left(\overline{Z}^{(q)}(b)+\frac{\psi^{\prime}(0+)}{q}\right)
\nonumber\\
\hspace{-0.3cm}&
=&\hspace{-0.3cm}
-\overline{Z}^{(q)}(x)-\frac{\psi^{\prime}(0+)}{q}+Z^{(q)}(x)\lim\limits_{b\rightarrow\infty}\frac{W^{(q)}(b)}{ W^{(q)\prime}(b)}
\nonumber\\
\hspace{-0.3cm}&
=&\hspace{-0.3cm}
-\overline{Z}^{(q)}(x)-\frac{\psi^{\prime}(0+)}{q}+\frac{Z^{(q)}(x)}{\Phi(q)},\nonumber
\end{eqnarray}
which coincides with the results obtained by letting $a\rightarrow\infty$ in (4.4) of Avram et al. (2007).
\end{rem}

\vspace{0.2cm}
The next result gives an expression of the Laplace transform of the accumulated capital injections until time $\kappa_{b}^{+}$.

\vspace{0.2cm}
\begin{thm}\label{3.4}
For any $\theta>0$ and $x\leq b$, we have
\begin{eqnarray}\label{lap.of.exp.tot.dis.cap.inj.}
\overline{V}_\xi(x;b)
\hspace{-0.3cm}&
=&\hspace{-0.3cm}
\exp\left(-\int_{x}^{b}\left(\theta-\frac{\psi(\theta)W(\overline{\xi}(y))}{Z(\overline{\xi}(y),\theta)}\right)
\mathrm{d}y\right).
\end{eqnarray}
\end{thm}

\begin{proof}[Proof:]\,\,\,
For $q>0$ and $x\leq b$, we have
\begin{eqnarray}\label{add.1}
\hspace{-0.3cm}&&\hspace{-0.3cm}\overline{V}_\xi(x;b):=
\mathbb{E}_{x}\left(\mathrm{e}^{-\theta R(\kappa_{b}^{+})}\right)
\nonumber\\
\hspace{-0.3cm}&=&\hspace{-0.3cm}
\mathbb{E}_{x}\left(\mathrm{e}^{-\theta\left(\xi(\overline{X}(\tau_{\xi}))-X(\tau_{\xi})\right)
-\theta\left(R(T_{1})-R(\tau_{\xi})\right)
-\theta\left(R(\kappa_{b}^{+})-R(T_{1})\right)
};\tau_{\xi}<\kappa_{b}^{+}
\right)
+
\mathbb{P}_{x}\left(\tau_{\xi}>\tau_{b}^{+}\right)
\nonumber\\
\hspace{-0.3cm}&=&\hspace{-0.3cm}
\mathbb{E}_{x}\left(\mathrm{e}^{-\theta\left(\xi(\overline{X}(\tau_{\xi}))-X(\tau_{\xi})\right)}
\mathrm{e}^{
-\theta\left(R(T_{1})-R(\tau_{\xi})\right)}
\mathbb{E}_{x}\left(\left.\mathrm{e}^{
-\theta\left(R(\kappa_{b}^{+})-R(T_{1})\right)
}\right|\mathcal{F}_{T_{1}}
\right)
;\tau_{\xi}<\kappa_{b}^{+}\right)
\nonumber\\
\hspace{-0.3cm}&&\hspace{-0.3cm}
+
\exp\left(-\int_{x}^{b}\frac{W^{\prime}(\overline{\xi}\left(z\right))}
{W(\overline{\xi}\left(z\right))}\mathrm{d}z\right)
\nonumber\\
\hspace{-0.3cm}&=&\hspace{-0.3cm}
\mathbb{E}_{x}\left(\mathrm{e}^{-\theta\left(\xi(\overline{X}(\tau_{\xi}))-X(\tau_{\xi})\right)}
\overline{V}_\xi\left(\overline{X}(\tau_{\xi}),b\right)
\overline{V}_0 (0; \overline{\xi}(\overline{X}(\tau_{\xi})))
;\tau_{\xi}<\tau_{b}^{+}\right)
\nonumber\\
\hspace{-0.3cm}&&\hspace{-0.3cm}
+
\exp\left(-\int_{x}^{b}\frac{W^{\prime}(\overline{\xi}\left(z\right))}
{W(\overline{\xi}\left(z\right))}\mathrm{d}z\right).
\end{eqnarray}
By (24) of Albrecher et al. (2016), we have
\begin{eqnarray}\label{add.2}
\overline{V}_0 (0; \overline{\xi}(\overline{X}(\tau_{\xi})))
\hspace{-0.3cm}&=&\hspace{-0.3cm}
\left.\mathbb{E}_{}\left(\mathrm{e}^{-\theta R_{0}(\sigma_{z}^{+})}\right)\right|_{z=\overline{\xi}(\overline{X}(\tau_{\xi}))}
\nonumber\\
\hspace{-0.3cm}&=&\hspace{-0.3cm}
\left.\mathbb{E}_{}\left(\mathrm{e}^{-\theta R_{0}(\sigma_{z}^{+})};\sigma_{z}^{+}<\infty\right)\right|_{z=\overline{\xi}(\overline{X}(\tau_{\xi}))}
\nonumber\\
\hspace{-0.3cm}&=&\hspace{-0.3cm}
\frac{Z(0,\theta)}{Z(\overline{\xi}(\overline{X}(\tau_{\xi})),\theta)}
\nonumber\\
\hspace{-0.3cm}&=&\hspace{-0.3cm}
\frac{1}{Z(\overline{\xi}(\overline{X}(\tau_{\xi})),\theta)},
\end{eqnarray}
where $R_{0}(t):=-\underline{X}(t)\wedge 0$, and for the second equality of \eqref{add.2} we need the fact that if $\psi^{\prime}(0+)\geq0$, then $\lim\limits_{t\rightarrow\infty}\overline{X}(t)=\infty$, and  we have \[\lim\limits_{t\rightarrow\infty}\sup_{s\in[0,t]}Y(s)\geq\lim\limits_{t\rightarrow\infty}\overline{X}(t)=\infty,\]
which implies $\sigma_{z}^{+}<\infty$ \,\, $\mathbb{P}_{x}$-a.s.; if $\psi^{\prime}(0+)<0$ and $\sigma_{z}^{+}=\infty$ a.s., then we  have
$R_{0}(\sigma_{z}^{+})=\infty$ because $\lim\limits_{t\rightarrow\infty}\underline{X}(t)=-\infty$ when $\psi^{\prime}(0+)<0$. That is to say, either  $\psi^{\prime}(0+)\geq0$ or $\psi^{\prime}(0+)<0$, we always have
\[\mathbb{E}_{x}\left(\mathrm{e}^{-\theta R_{0}(\sigma_{z}^{+})}\right)=
\mathbb{E}_{x}\left(\mathrm{e}^{-\theta R_{0}(\sigma_{z}^{+})};\sigma_{z}^{+}<\infty\right).\]

Using \eqref{13}, we have
\begin{eqnarray}\label{add.3}
\hspace{-0.3cm}&&\hspace{-0.3cm}
\mathbb{E}_{x}\left(\mathrm{e}^{\theta X(\tau_{\xi})}
\frac{\mathrm{e}^{-\theta\xi(\overline{X}(\tau_{\xi}))}\overline{V}_\xi\left(\overline{X}(\tau_{\xi}),b\right)}
{Z(\overline{\xi}(\overline{X}(\tau_{\xi})),\theta)}
;\tau_{\xi}<\tau_{b}^{+}\right)
\nonumber\\
\hspace{-0.3cm}&=&\hspace{-0.3cm}\int_{x}^{b}
\frac{\overline{V}_\xi\left(s,b\right)}
{Z(\overline{\xi}(s),\theta)}\,
\mathrm{e}^{-\int_{x}^{s}\frac{W^{\prime}(\overline{\xi}\left(z\right))}
{W(\overline{\xi}\left(z\right))}\mathrm{d}z}
\left(\frac{W^{\prime}(\overline{\xi}(s))}{W(\overline{\xi}(s))}Z(\overline{\xi}(s),\theta)-\theta Z(\overline{\xi}(s),\theta)+\psi(\theta)W(\overline{\xi}(s))\right)
\mathrm{d}s.
\end{eqnarray}
Applying both (\ref{add.2}) and (\ref{add.3}) to (\ref{add.1}), we have
\begin{eqnarray}\label{add.4}
\overline{V}_\xi(x;b)
\hspace{-0.3cm}&=&\hspace{-0.3cm}
\exp\left(-\int_{x}^{b}\frac{W^{\prime}(\overline{\xi}\left(z\right))}
{W(\overline{\xi}\left(z\right))}\mathrm{d}z\right)
+
\mathbb{E}_{x}\left(\mathrm{e}^{\theta X(\tau_{\xi})}
\frac{\mathrm{e}^{-\theta\xi(\overline{X}(\tau_{\xi}))}\overline{V}_\xi\left(\overline{X}(\tau_{\xi}),b\right)}
{Z(\overline{\xi}(\overline{X}(\tau_{\xi})),\theta)}
;\tau_{\xi}<\tau_{b}^{+}\right)
\nonumber\\
\hspace{-0.3cm}&=&\hspace{-0.3cm}
\exp\left(-\int_{x}^{b}\frac{W^{\prime}(\overline{\xi}\left(z\right))}
{W(\overline{\xi}\left(z\right))}\mathrm{d}z\right)
\nonumber\\
\hspace{-0.3cm}&&\hspace{-0.7cm}
+\int_{x}^{b}
\frac{\overline{V_\xi}\left(s; b\right)}
{Z(\overline{\xi}(s),\theta)}\,
\mathrm{e}^{-\int_{x}^{s}\frac{W^{\prime}(\overline{\xi}\left(z\right))}
{W(\overline{\xi}\left(z\right))}\mathrm{d}z}
\left(\frac{W^{\prime}(\overline{\xi}(s))}{W(\overline{\xi}(s))}Z(\overline{\xi}(s),\theta)-\theta Z(\overline{\xi}(s),\theta)+\psi(\theta)W(\overline{\xi}(s))\right)
\mathrm{d}s.
\end{eqnarray}
Taking derivatives on  both sides of (\ref{add.4}) with respect to $x$, we have
\begin{eqnarray}\label{add.5}
\overline{V}_\xi^{\prime}(x;b)
\hspace{-0.3cm}&=&\hspace{-0.3cm}
\frac{W^{\prime}(\overline{\xi}\left(x\right))}
{W(\overline{\xi}\left(x\right))}\overline{V}_\xi(x;b)
\nonumber\\
\hspace{-0.3cm}&&\hspace{-0.3cm}
-
\frac{\overline{V}_\xi\left(x; b\right)}
{Z(\overline{\xi}(x),\theta)}
\left(\frac{W^{\prime}(\overline{\xi}(x))}{W(\overline{\xi}(x))}Z(\overline{\xi}(x),\theta)-\theta Z(\overline{\xi}(x),\theta)+\psi(\theta)W(\overline{\xi}(x))\right)
\nonumber\\
\hspace{-0.3cm}&=&\hspace{-0.3cm}
\left(\theta-\frac{\psi(\theta)W(\overline{\xi}(x))}{Z(\overline{\xi}(x),\theta)}\right)
\overline{V}_\xi(x;b)
.
\end{eqnarray}
Solving (\ref{add.5}) with boundary condition $\overline{V}(b;b)=1$ we obtain (\ref{lap.of.exp.tot.dis.cap.inj.}).
\end{proof}

\vspace{0.2cm}
\begin{rem}
Letting $\xi(x)\equiv0$ in (\ref{lap.of.exp.tot.dis.cap.inj.}), we get
\begin{eqnarray}\label{expr.of.exp.tot.dis.cap.inj.1}
\overline{V}_0(x;b)
\hspace{-0.3cm}&
=&\hspace{-0.3cm}
\exp\left(-\int_{x}^{b}
\left(\theta-\frac{\psi(\theta)W(y)}{Z(y,\theta)}\right)
\mathrm{d}y\right)
\nonumber\\
\hspace{-0.3cm}&
=&\hspace{-0.3cm}
\exp\left(-\int_{x}^{b}
\frac{Z^{\prime}(y,\theta)}{Z(y,\theta)}
\mathrm{d}y\right)
=
\frac{Z(x,\theta)}{Z(b,\theta)},
\end{eqnarray}
which recovers  (24) of Albrecher et al. (2016).
\end{rem}

\vspace{0.2cm}
\begin{rem}
By Theorem \ref{3.4}, one can deduce that,
for $q>0$ and $x\leq b$
\begin{eqnarray}\label{lap.of.exp.tot.dis.cap.inj.1}
\mathbb{E}_{x}\left(\mathrm{e}^{-q\kappa_{b}^{+}-\theta R(\kappa_{b}^{+})}\right)
\hspace{-0.3cm}&
=&\hspace{-0.3cm}
\mathbb{E}_{x}\left(\mathrm{e}^{-q\kappa_{b}^{+}-\theta R(\kappa_{b}^{+})};\kappa_{b}^{+}<\infty\right)
\nonumber\\
\hspace{-0.3cm}&
=&\hspace{-0.3cm}
\mathbb{E}_{x}\left(\mathrm{e}^{-\theta R(\kappa_{b}^{+})};\kappa_{b}^{+}<\infty, \kappa_{b}^{+}<e_{q}\right)
\nonumber\\
\hspace{-0.3cm}&
=&\hspace{-0.3cm}
\exp\left(-\int_{x}^{b}\left(\theta-\frac{\psi_{q}(\theta)W^{(q)}(\overline{\xi}(y))}{Z^{(q)}(\overline{\xi}(y),\theta)}\right)\mathrm{d}y\right),\nonumber
\end{eqnarray}
where $e_{q}$ is an exponential random variable with rate $q$ independent of $X$, and $\psi_{q}(\theta)=\psi(\theta)-q$, $W^{(q)}$ and $Z^{(q)}$ correspond to the Laplace exponent and scale functions of the L\'{e}vy process killed at rate $q$.
\end{rem}

\vspace{0.2cm}
\begin{rem}
It is easy to see that Theorem \ref{3.3} and Theorem \ref{3.4} agree with each other in some special cases.
By Theorem  \ref{3.3} we have
\begin{eqnarray}\label{60}
\mathbb{E}_{x}\left(R(\kappa_{b}^{+})\right)=
\mathbb{E}_{x}\left(\int_{0}^{\kappa_{b}^{+}}\mathrm{d}R(t)\right)=\lim\limits_{q\downarrow0}V_{\xi}(x; b)\hspace{-0.3cm}&
=&\hspace{-0.3cm}\int_{x}^{b}\left(1-\psi^{\prime}(0+)W(\overline{\xi}(y))\right)
\mathrm{d}y,
\end{eqnarray}
where $\lim\limits_{q\downarrow0}Z^{(q)}(x)=1$ is used.
In the meanwhile, it holds that
\begin{eqnarray}\label{61}
\mathbb{E}_{x}\left(-R(\kappa_{b}^{+})\right)=
\lim\limits_{\theta\downarrow0}\frac{\mathrm{d}}{\mathrm{d}\theta}
\overline{V}_{\xi}(x;b).
\end{eqnarray}
By Theorem  \ref{3.4} we have
\begin{eqnarray}\label{62}
\frac{\mathrm{d}}{\mathrm{d}\theta}
\overline{V}_{\xi}(x;b)\hspace{-0.3cm}&
=&\hspace{-0.3cm}
\exp\left(-\int_{x}^{b}\left(\theta-\frac{\psi(\theta)W(\overline{\xi}(y))}{Z(\overline{\xi}(y),\theta)}\right)\mathrm{d}y\right)
\nonumber\\
\hspace{-0.3cm}&
&\hspace{-0.3cm}
\times
\int_{x}^{b}-
\left(1-\frac{\psi^{\prime}(\theta)W(\overline{\xi}(y))}{Z(\overline{\xi}(y),\theta)}
+\frac{\psi(\theta)W(\overline{\xi}(y))}{[Z(\overline{\xi}(y),\theta)]^{2}}
\frac{\mathrm{d}}{\mathrm{d}\theta}Z(\overline{\xi}(y),\theta)\right)\mathrm{d}y,\nonumber
\end{eqnarray}
with
$$
\frac{\mathrm{d}}{\mathrm{d}\theta}Z(x,\theta)=xZ(x,\theta)-\mathrm{e}^{\theta x}\psi^{\prime}(\theta)\int_{0}^{x}\mathrm{e}^{-\theta z}W(z)\mathrm{d}z
+\mathrm{e}^{\theta x}\psi(\theta)\int_{0}^{x}z\mathrm{e}^{-\theta z}W(z)\mathrm{d}z,
$$
by the definition of $Z(x,\theta)$.
Hence,
\begin{eqnarray}\label{63}
\lim\limits_{\theta\downarrow0}\frac{\mathrm{d}}{\mathrm{d}\theta}Z(x,\theta)
=x-\psi^{\prime}(0+)\int_{0}^{x}W(z)\mathrm{d}z,\nonumber
\end{eqnarray}
and
\begin{eqnarray}\label{64}
\lim\limits_{\theta\downarrow0}\frac{\mathrm{d}}{\mathrm{d}\theta}
\overline{V}_{\xi}(x;b)\hspace{-0.3cm}&
=&\hspace{-0.3cm}
-\int_{x}^{b}
\left(1-\psi^{\prime}(0+)W(\overline{\xi}(y))\right)\mathrm{d}y.
\end{eqnarray}
Plugging (\ref{64}) into (\ref{61}) we recover (\ref{60}).

\end{rem}

\vspace{1cm}
\section*{Acknowledgements}
The authors are grateful to an anonymous referee  for very helpful comments.
Wenyuan Wang acknowledges the financial support from the National Natural Science Foundation of China (No.11601197) and the Program for New Century Excellent Talents in Fujian Province University. He also thanks Concordia University where this paper was finished during his visit.
Xiaowen Zhou acknowledges the financial support from  NSERC (RGPIN-2016-06704) and National Natural Science Foundation of China (No.11771018) .

\end{document}